\documentclass [11pt,twoside]{amsart}
\usepackage[top=1.2in, bottom=1.2in, left=1.5in, right=1.5in]{geometry}
\usepackage{tikz}
\usepackage[active]{srcltx}
\usepackage[curve]{xypic}
\usepackage{graphicx}
\usepackage{mathrsfs}
\usepackage{color}
\usepackage{xcolor}
\usepackage{amssymb,amsmath,amsthm,amsfonts}
\usepackage{enumerate}
\usepackage{mathtools}
\usepackage{physics}
\usepackage{mathrsfs}
\usepackage{thmtools}
\usepackage{hyperref}

\allowdisplaybreaks[4]



\numberwithin{equation}{section}
\numberwithin{figure}{section}

\theoremstyle{plain}
\newtheorem{thm}{Theorem}[section]
\theoremstyle{plain}
\newtheorem*{thm*}{Theorem}

\theoremstyle{definition}

\theoremstyle{definition}
\newtheorem*{defn*}{Defnition}

\theoremstyle{remark}
\newtheorem{rem}[thm]{Remark}
\theoremstyle{remark}
\newtheorem*{rem*}{Remark}

\theoremstyle{plain}
\newtheorem{lem}[thm]{Lemma}

\theoremstyle{plain}
\newtheorem{prop}[thm]{Proposition}

\theoremstyle{plain}

\theoremstyle{plain}
\newtheorem{construction}[thm]{Construction}

\begin{document}

\title[Hausdorff Dimension of Sets of Continued Fractions]{Hausdorff Dimension of Sets of Continued Fractions with Bounded Odd and Even Order Partial Quotients}

\author{\noindent Yuefeng Tang}

\keywords{continued fraction, bounded partial quotients, Hausdorff dimension.}

\begin{abstract}
	We study the continued fractions with bounded odd/even-order partial quotients. In particular, we investigate the sizes of the sets of continued fractions whose odd-order partial quotients are equal to 1. We demonstrate that the sum and the product of two sets of continued fractions whose odd-order partial quotients are equal to 1 both contain non-empty intervals. Our work compliments the results of Han\v{c}l and Turek \cite{MR4632209} on the set of continued fractions whose even-order partial quotients are equal to 1. Furthermore, we determine the Hausdorff dimensions of the sets of continued fractions whose odd-order partial quotients are equal to 1 and even-order partial quotients are growing at an exponential rate, a super-exponential rate, and in general a positive function rate.
\end{abstract}

\maketitle

\tableofcontents

\section{Introduction}\label{1}

\subsection{Background and history}

It is known that every irrational number $x\in[0,1)$ admits a unique infinite continued fraction induced by Gauss transformation $T\colon [0,1)\rightarrow[0,1)$ given by
\begin{equation*}
	T(0)\coloneqq 0,\quad T(x)\coloneqq\frac{1}{x}\ (\textnormal{mod}\ 1)\quad \text{for}\ x\in(0,1). 
\end{equation*}
We write shortly $x=[a_0(x);a_1(x),a_2(x),\ldots]$ for the continued fraction expansion of $x$ as follows:
\begin{equation*}
	x=a_0(x)+\cfrac{1}{a_1(x)+\cfrac{1}{a_2(x)+\cdots}},
\end{equation*}
where $a_0(x)=[x]$, $a_1(x)=[\frac{1}{x}]$, and $a_n(x)=a_1(T^{n-1}(x))$ for $n\geq 2$ are called \textit{partial quotients} of $x$. For each real number $x$, the sequence of partial quotients $\{a_n(x)\}_{n\geq1}$ is either infinite, when $x$ is irrational, or finite, when $x$ is rational. Truncating the continued fractions at its $k$th term, we obtain an irreducible rational number
\begin{equation*}
	\frac{p_k(x)}{q_k(x)}=[a_0(x);a_1(x),a_2(x),\ldots,a_k(x)],
\end{equation*}
called the $k$th \textit{convergent} of $x$.

The metrical theory of continued fractions, which is one of the major subjects in the study of continued fractions, has close connections with many fields, such as dynamical system, fractal geometry, ergodic theory, number theory, algebraic geometry etc. The metric property of sets of continued fractions, in particular the fractal dimensions of sets arising in continued fractions, is of vital position. 

Back in 1947, Hall \cite{hall} studied the problem how to express a real number as a sum or a product of continued fractions with bounded partial quotients. He proved that any real number $x$ can be represented by a sum of two continued fractions with partial quotients bounded by 4, i.e.,
\begin{equation*}
	x=[a_0;a_1,a_2,a_3,\ldots]+[b_0;b_1,b_2,b_3,\ldots],
\end{equation*}
where $a_j,b_j\leq4$ for all $j\geq1$. He also showed that any real number $x>1$ can be represented by a product of two continued fractions whose partial quotients are not more than 4. Hall's result inspired a series of improvements and generalizations, see \cite{astels1,astels3,astels2,cusick,divis,hlavka}. 

In 2023, Han\v{c}l and Turek \cite{MR4632209} studied the set of continued fractions with bounded even-order partial quotients, i.e.,
\begin{equation*}
	\{[a_0;a_1,a_2,a_3,\ldots]\colon a_{2k}\leq c,\ \text{for all}\ k\in\mathbb{N}\},
\end{equation*}
where $c>0$ is a fixed constant. Their ideas come from the study of the approximation to irrational numbers strictly from one side, which has a close relationship with the boundedness of the subsequences $\{a_{2j-1}\}_{j=1}^\infty$ and $\{a_{2j}\}_{j=1}^\infty$. Han\v{c}l and Turek \cite{MR4632209} proved that any real number $x$ can be represented by a sum of two continued fractions whose even-order partial quotients are equal to 1, and any positive real number $x>0$ can be represented by a product of two continued fractions whose even-order partial quotients are equal to 1. Motivated by their work, we study the property of the odd conjugation cases, i.e., the set of continued fractions whose odd-order partial quotients are equal to 1. Our first main results are as follows.

\begin{thm}\label{main1}
	Any real number $x\in\mathbb{Z}+[1/2,1]$ can be represented by a sum of two continued fractions whose odd-order partial quotients are equal to $1$, i.e.,\begin{equation}
		x=[a_0;1,a_2,1,a_4,1,a_6,\ldots]+[b_0;1,b_2,1,b_4,1,b_6,\ldots], 
	\end{equation}
	where $a_0,b_0\in\mathbb{Z}$, $a_i,b_i\in\mathbb{N}$, $i>0$.
\end{thm}

\begin{thm}\label{main2}
	Any positive real number $x\in\mathbb{Z}+[1/2,1]$ can be represented by a product of two continued fractions whose odd-order partial quotients are equal to $1$, i.e.,\begin{equation}
		x=[a_0;1,a_2,1,a_4,1,a_6,\ldots]\cdot[b_0;1,b_2,1,b_4,1,b_6,\ldots],
	\end{equation}
	where $a_0,b_0\in\mathbb{N}$, $a_i,b_i\in\mathbb{N}$, $i>0$.
\end{thm}


Han\v{c}l and Turek \cite{MR4632209} also studied the sets of continued fractions whose even-order partial quotients are bounded or equal to a given positive integer. It is obvious that these sets are of zero Lebesgue measure, and the research on Hausdorff dimension of such exceptional sets receives a lot of attention. Borel--Bernstein theorem gives a rough way of determining the Lebesgue measure of a certain kind of sets of continued fractions, while the fractional dimensional theory provides detailed information of the size and complexity of such sets.

\begin{thm}[Borel--Bernstein]
	Let $\phi$ be an arbitrary positive function defined on $\mathbb{N}$ and
	\begin{equation*}
		E(\phi)\coloneqq\{x\in[0,1)\colon a_n(x)\geq\phi(n)\ \textnormal{i.m.}\ n\}.
	\end{equation*}
	Then $\mathcal{L}(E(\phi))$ is null or full, according as the series $\sum_{n=1}^\infty\frac{1}{\phi(n)}$ converges or diverges, where $\mathcal{L}$ denotes Lebesgue measure.
\end{thm}

\noindent The well-known Borel--Bernstein theorem, which is a 0-1 law with respect to the Lebesgue measure, plays an important role in the metrical theory of continued fractions. It was proved originally by Borel \cite{Borel1909LesPD} and Bernstein \cite{MR1511668}. Further details were provided later by Borel \cite{MR1511714}. For proofs, we refer to Khintchine \cite{MR0161834}, Hardy and Wright \cite{MR0568909}.

The first published work on fractal dimension of exceptional sets in continued fractions is known as the paper by Jarn\'ik \cite{Jarnik1928-1929}, in which the author studied the sets of badly approximable numbers with bounded partial quotients. After that, Good \cite{MR0004878} studied the set $\{x\in[0,1)\colon a_n(x)\rightarrow\infty\}$, and provided an estimation of upper bound of the fractal dimension of set $\{x\in[0,1)\colon a_n(x)\geq\phi(n)\ \textnormal{i.m.}\ n\}$, where {\textnormal{i.m.}} stands for infinitely many. After that, many authors focus on the growth rate of partial quotients from various aspects, and due to the flourishing development of dynamical system, many new tools are brought to the study of the sets of continued fractions (see \cite{MR1102677,MR1449135,MR3099298}).

One classical method to study the fractal dimension of the sets of continued fractions is to restrict the growing rates of partial quotients. \L uczak \cite{MR1464375} and Feng, Wu, Liang and Tseng \cite{MR1464376} proved that the Hausdorff dimension of the set
\begin{equation*}
	\{x\in I\colon a_{n}(x)\geq c^{b^{n}}\ \textnormal{i.m.}\ n\} \quad(b,c>1)
\end{equation*}
is $1/(1+b)$. Shang and Fang \cite{shangfang} recently gave a new simple proof for the upper bound of the Hausdorff dimension of the above set. Wang and Wu \cite{MR2419924} completely determined the Hausdorff dimension of the set
\begin{equation*}
	\{x\in I\colon a_{n}(x)\geq B^{n}\ \textnormal{i.m.}\ n\} \quad(B>1),
\end{equation*}
and they obtained the Hausdorff dimension of exceptional sets of Borel-Bernstein theorem:
\begin{equation*}
	\{x\in[0,1)\colon a_n(x)\geq\phi(n)\ \textnormal{i.m.}\ n\},
\end{equation*}
where $\phi$ is an arbitrary positive function on $\mathbb{N}$, such that $\sum_{n\geq1}\frac{1}{\phi(n)}<\infty$.
Bugeaud, Robert and Hussain \cite{bugeaud2023metrical} gave the complex analogue of the result of Wang and Wu. They determined the Hausdorff dimension of the set
\begin{equation*}
	\{z\in\mathbb{C}\colon \abs{a_n(z)}\geq\phi(n)\ \textnormal{i.m.}\ n\},
\end{equation*}
where $a_n(z)$ denotes the $n$th partial quotient in the Hurwitz continued fraction of a complex number $z$. 

In the second half part of the present paper, we determine the Hausdorff dimensions of sets of continued fractions whose odd-order partial quotients are equal to 1 and even-order partial quotients are growing at an exponential rate, a super-exponential rate, and in general a positive function rate. 

\subsection{Statement of results}

We define the set
\begin{equation*}
	I_{even}\coloneqq\{x=[0;a_1,a_2,a_3\ldots]\in I:a_{2k+1}=1,\ \text{for all}\ k\in \mathbb{N}\}.
\end{equation*}
To simplify our notation, we denote $q_{2n}(1,a_2,1,a_4,\ldots,1,a_{2n})=q_{2n}$. For the other notations, we use a similar way to simplify and assume the elements of odd-order are equal to 1 by default. For any $B>1$, we write 
\begin{equation*}
	s_B\coloneqq \lim\limits_{n\rightarrow\infty}\inf\left\{\rho\geq0\colon \sum_{a_2,a_4,\ldots,a_{2n}\in\mathbb{N}}\frac{1}{(B^{2n}q_{2n}^2)^\rho}\leq1\right\}.
\end{equation*}

In the following sections, we restrict ourselves to the case that the odd-order partial quotients are equal to 1 and determine the Hausdorff dimension of the sets of continued fractions whose even-order partial quotients are growing at an exponential rate:
\begin{equation*}
	F(B)=\big\{x\in I_{even}\colon a_{2n}(x)\geq B^{2n}\ \textnormal{i.m.}\ n\big\}	\quad(B>1),
\end{equation*}
the sets of continued fractions whose even-order partial quotients are growing at a super-exponential rate
\begin{equation*}
	E(b,c)\coloneqq\big\{x\in I_{even}\colon a_{2n}(x)\geq c^{b^{2n}}\ \textnormal{i.m.}\ n\big\}\quad (b,c>1),
\end{equation*}
and
\begin{equation*}
	\widetilde{E}(b,c)\coloneqq\big\{x\in I_{even}\colon a_{2n}(x)\geq c^{b^{2n}}\ \textnormal{for all}\ n\geq1\big\}\quad (b,c>1),
\end{equation*}
and the set of continued fractions whose even-order partial quotients are growing at a general positive function rate:
\begin{equation*}
	F(\phi)\coloneqq\big\{x\in I_{even}\colon a_{2n}(x)\geq \phi(n)\ \ \textnormal{i.m.}\ n\big\},
\end{equation*}
where $\phi$ is an arbitrary positive function on $\mathbb{N}$, such that $\sum_{n\geq1}\frac{1}{\phi(n)}<\infty$.

We use $\dim_H$ to denote the Hausdorff dimension.

\begin{thm}\label{main3}
	For any $B>1$, we have $\dim_H F(B)=s_B$.
\end{thm}

It should be mentioned that $s_B$ is decreasing to $1/2$ as $B$ tends to $\infty$.

\begin{thm}\label{main4}
	For any $b,c>1$, we have
	\begin{equation*}
		\dim_H E(b,c)=\dim_H \widetilde{E}(b,c)=\dfrac{1}{1+b^2}.
	\end{equation*}
\end{thm}

As a generalization of both results above, we have the following theorem.

\begin{thm}\label{main5}
	Let $\phi$ be an arbitrary positive function defined on $\mathbb{N}$. Suppose $\liminf\limits_{n\rightarrow\infty}\frac{\log\phi(n)}{2n}=\log B$. If $B=\infty$, suppose $\liminf\limits_{n\rightarrow\infty}\frac{\log\log\phi(n)}{2n}=\log b$.
	\begin{itemize}
		\item If $B=1$, then $s_1\leq\dim_H F(\phi)\leq\dim_H I_{even}$.
		\item If $1<B<\infty$, then $\dim_H F(\phi)=s_B$.
		\item If $B=\infty$ and $b=1$, then $\dim_H F(\phi)=\frac{1}{2}$.
		\item If $B=\infty$ and $1<b<\infty$, then $\dim_H F(\phi)=\frac{1}{1+b^2}$.
		\item If $B=\infty$ and $b=\infty$, then $\dim_H F(\phi)=0$.
	\end{itemize}
\end{thm}

Our approach of computing the Hausdorff dimensions of $E(b,c)$ and $\widetilde{E}(b,c)$ involves ideas from \L uczak \cite{MR1464375}, Feng, Wu, Liang and Tseng \cite{MR1464376} and relies on Falconer \cite[Example 4.6]{MR1102677}. The method of determining the Hausdorff dimensions of $F(B)$ and $F(\phi)$ is motivated by Wang and Wu \cite{MR2419924} where we make full use of the mass distribution principle.

\subsubsection*{\textbf{\emph{Structure of the paper.}}}

The present paper is organized as follows. In Section \ref{2}, we introduce some notations, define necessary concepts, present some basic properties of continued fractions, and construct a useful auxiliary function. In Sections \ref{3} and \ref{4}, we separately prove that the sum and the product of $I_{even}$ contain intervals. Section \ref{5} is devoted to the determination of the Hausdorff dimension of $F(B)$. In Section \ref{6}, we compute the Hausdorff dimension of $E(b,c)$ and $\widetilde{E}(b,c)$. In Section \ref{7}, we fully describe the Hausdorff dimension of $F(\phi)$.

\subsubsection*{\textbf{\emph{Acknowledgment.}}}

The author wishes to express sincere appreciation to Professor Lingmin Liao who critically read the paper and made numerous helpful suggestions.

\section{Preliminaries}\label{2}

\subsection{Basic notations and concepts}

We will use the following notations.

\begin{itemize}
	\item[] $[x]$\quad the integer part of a real number $x$.
	\item[] $I$\quad the interval $[0,1)$.
	\item[] $\abs{\cdot}$\quad the Lebesgue measure of a set.
	\item[] $\mathcal{H}^t$\quad $t$-dimensional Hausdorff measure.
	\item[] $\#$\quad the cardinality of a set.
	\item[] $\textnormal{cl}$\quad the topological closure of a set.
\end{itemize}


Let $F$ be a subset of $\mathbb{R}^n$. For $\delta>0$, we call a countable (or finite) sequence of sets $\{U_i\}$ a \textit{$\delta$-cover} of set $F$, if $0<\abs{U_i}\leq\delta$ and $F\subset\cup_iU_i$. Let $t\geq0$, we define
\begin{equation*}
	\mathcal{H}^t_\delta(F)\coloneqq\inf\Big\{\sum_{i=1}^{\infty}\abs{U_i}^t\colon\{U_i\}\ \text{is a $\delta$-cover of $F$}\Big\}.
\end{equation*}
Notice that $\mathcal{H}^t_\delta(F)$ is a decreasing function with respect to $\delta$. We define the \textit{$t$-dimensional Hausdorff measure} of set $F$ by
\begin{equation*}
	\mathcal{H}^t(F)=\lim_{\delta\rightarrow0}\mathcal{H}^t_{\delta}(F)=\sup_{\delta>0}\mathcal{H}^t_{\delta}(F).
\end{equation*}
Notice that for any set $F$, there exists a critical point of $t$ at which $\mathcal{H}^t(F)$ jumps from $\infty$ to $0$. The critical point 
\begin{equation*}
	\dim_H(F)=\inf\{t>0\colon\mathcal{H}^t(F)=0\}=\sup\{t>0\colon\mathcal{H}^t(F)=\infty\}
\end{equation*}
is called the \textit{Hausdorff dimension} of $F$.

We also present the mass distribution principle (see \cite{MR1102677}), which will be used in Section \ref{3} to determined the lower bound of Hausdorff dimension of $F(B)$.

\begin{lem}[mass distribution principle]\label{dp}
	Let $E\subset \mathbb{R}^n$ be a Borel set, $\mu$ a measure with $\mu(E)>0$. If for any $x\in E$,
	\begin{equation*}
		\liminf_{r\rightarrow 0}\frac{\log{\mu(B(x,r))}}{\log{r}}\geq s,
	\end{equation*} 
	where $B(x,r)$ denotes the open ball of radius $r$ centered at $x$, then we have $\dim_H E\geq s$.
\end{lem}


Let $\frac{p_n}{q_n}=[a_0;a_1,a_2,\ldots,a_n]$ be the $n$th convergent of the real number $x=[a_0;a_1,a_2,a_3,\ldots]$, where $p_k,q_k,0\leq k\leq n$, are defined recursively by the following relations (see \cite[Theorem 1]{MR0161834})
\begin{align*}
	&p_{-1}=1;\quad	p_0=a_0;\quad	p_k=a_kp_{k-1}+p_{k-2},\quad	1\leq k\leq n,	\\
	&q_{-1}=0;\quad	q_0=1;\quad	q_k=a_kq_{k-1}+q_{k-2},\quad	1\leq k\leq n.	
\end{align*}
Further, we have the identities (see \cite[Theorem 2]{MR0161834})
\begin{align}
	&\frac{p_{n-1}}{q_{n-1}}-\frac{p_n}{q_n}=\frac{(-1)^n}{q_nq_{n-1}},\quad n\geq1,\label{pre1}	\\
	&\frac{p_{n-2}}{q_{n-2}}-\frac{p_n}{q_n}=\frac{(-1)^{n-1}a_n}{q_nq_{n-2}},\quad n\geq2.\label{pre2}
\end{align}
The denominators of the consecutive convergents satisfy (see \cite[Theorem 6]{MR0161834})
\begin{equation}\label{pre3}
	\frac{q_n}{q_{n-1}}=[a_n;a_{n-1},a_{n-2},\ldots,a_1],\quad	n\geq1.
\end{equation}
The numerators satisfy a similar formula (see \cite{MR4632209}) that if $n\geq1$, $a_0\geq0$ and $n+a_0\ne1$, then
\begin{equation}\label{pre4}
	\frac{p_n}{p_{n-1}}=\left\{\begin{aligned}
		&[a_n;a_{n-1},a_{n-2},a_{n-3},\ldots,a_2,a_1,a_0]&\text{for}\ a_0\ne0;	\\
		&[a_n;a_{n-1},a_{n-2},a_{n-3},\ldots,a_2]&\text{for}\ a_0=0.
	\end{aligned}\right.
\end{equation}

Given $a_1,a_2,\ldots,a_n,$ we call
\begin{equation*}
	I(a_1,a_2,\ldots,a_n)=\left\{\begin{aligned}
		&\left[\frac{p_n}{q_n},\frac{p_n+p_{n-1}}{q_n+q_{n-1}}\right),\quad\text{if $n$ is even},	\\
		&\left(\frac{p_n+p_{n-1}}{q_n+q_{n-1}},\frac{p_n}{q_n}\right],\quad\text{if $n$ is odd},	\\
	\end{aligned}\right.
\end{equation*} a \textit{cylinder of order $n$}.
In fact, $I(a_1,a_2,\ldots,a_n)$ represents the set of numbers in $I$ which have a continued fraction expansion beginning with $a_1,a_2,\ldots,a_n$, i.e.,
\begin{equation*}
	I(a_1,a_2,\ldots,a_n)=\{x\in I\colon a_1(x)=a_1,a_2(x)=a_2,\ldots,a_n(x)=a_n\}.
\end{equation*}
It is well known, (see \cite{MR0161834}), that \begin{equation}\label{pre5}
	\abs{I(a_1,a_2,\ldots,a_n)}=\frac{1}{q_n(q_n+q_{n-1})}.
\end{equation}

To compare two continued fractions, we need the following proposition (\cite{MR4632209}, Proposition 2.1):
\begin{prop}
	Let \begin{align*}
		\alpha&=[a_0;a_1,a_2,\ldots,a_{n-1},a_n,a_{n+1},\ldots],	\\
		\beta&=[b_0;b_1,b_2,\ldots,b_{n-1},b_n,b_{n+1},\ldots].
	\end{align*}
	Then 
	\begin{equation*}
		\alpha-\beta=\frac{(-1)^n([a_n;a_{n+1},\ldots]-[b_n;b_{n+1},\ldots])}{([a_n;a_{n+1},\ldots]q_{n-1}+q_{n-2})([b_n;b_{n+1},\ldots]q_{n-1}+q_{n-2})}.
	\end{equation*}
	If $a_n\ne b_n$, then
	\begin{equation*}
		\alpha<\beta\quad\Leftrightarrow\quad(n \text{ is even and } a_n<b_n)\text{ or } (n \text{ is odd and } a_n>b_n).
	\end{equation*}
	If $\alpha=[a_0;a_1,a_2,\ldots,a_n]$ and $\beta=[a_0;a_1,a_2,\ldots,a_n,a_{n+1},\ldots]$, then $\alpha<\beta$ if and only if $n$ is even.
\end{prop}

To do specific approximation on the size of sets of continued fractions, we also need the following lemmas.
\begin{lem}[{\cite[Lemma 2.1]{MR2215567}}]\label{lem1}
	For any $n\geq1$ and $1\leq k\leq n$, we have
	\begin{equation*}
		\frac{a_k+1}{2}\leq \frac{q_n(a_1,a_2,\ldots,a_n)}{q_{n-1}(a_1,\ldots,a_{k-1},a_{k+1},\ldots,a_n)}\leq a_k+1.
	\end{equation*}
\end{lem}

\begin{lem}[\cite{MR0161834}]\label{lem2}
	For any $n\geq1$ and $k\geq1$, we have
	\begin{align*}
		q_{n+k}(a_1,\ldots,a_n,a_{n+1},\ldots,a_{n+k})\geq q_n(a_1,\ldots,a_n)\cdot q_k(a_{n+1},\ldots,a_{n+k}),	\\
		q_{n+k}(a_1,\ldots,a_n,a_{n+1},\ldots,a_{n+k})\leq 2q_n(a_1,\ldots,a_n)\cdot q_k(a_{n+1},\ldots,a_{n+k}).
	\end{align*}
\end{lem}

\subsection{Construction of auxiliary function}

Let $\mathcal{A}\subseteq \mathbb{N}$ be a finite or infinite subset and $B>1$ be a fixed real number. For any integer $n\geq1$ and real number $\rho\geq0$, we define
\begin{equation*}
	f_{n,B}(\rho)=\sum_{a_2,a_4,\ldots,a_{2n}\in\mathcal{A}}\frac{1}{(B^{2n}q_{2n}^2)^\rho}.
\end{equation*}
It is easy to see that $f_{n,B}(\rho)$ is decreasing and $f_{n,B}(1)\leq1$. We define 
\begin{equation*}
	s_{n,B}(\mathcal{A})=\inf\{\rho\geq0:f_{n,B}(\rho)\leq1\}.
\end{equation*}

\begin{rem}\label{rem1}
	If $\mathcal{A}\subseteq\mathbb{N}$ is finite, we have $f_{n,B}(s_{n,B}(\mathcal{A}))=1$. If $\mathcal{A}\subseteq\mathbb{N}$ is infinite, we have $f_{n,B}(s_{n,B}(\mathcal{A}))\leq1$.
\end{rem}

\begin{lem}\label{lem3}
	We have $s_B(\mathcal{A})\coloneqq\lim_{n\rightarrow\infty}s_{n,B}(\mathcal{A})$ exists, and $0\leq s_B(\mathcal{A})\leq1$.
\end{lem}

\begin{proof}
	By Lemma \ref{lem2}, we have $q_{2(n+k)}\geq q_{2n}q_{2k}$ for any $n\geq1$ and $k\geq1$. Thus $f_{n+k}(\rho)\leq f_{n,B}(\rho)f_{k,B}(\rho)$ for any $\rho\geq0$. This implies 
	\begin{equation*}
		s_{n+k,B}(\mathcal{A})\leq\max\{s_{n,B}(\mathcal{A}),s_{k,B}(\mathcal{A})\}.
	\end{equation*}
	Observe that $s_{rn,B}(\mathcal{A})\leq s_{n,B}(\mathcal{A})$ for each $r\geq 1$. Thus, for any $r\geq1$ and $t\geq1$, \begin{equation*}
		s_{rn+tk,B}(\mathcal{A})\leq\max\{s_{n,B}(\mathcal{A}),s_{k,B}(\mathcal{A})\}.
	\end{equation*}
	We write $\bar{s}\coloneqq\limsup_{n\rightarrow\infty}s_{n,B}(\mathcal{A})$.
	
	\subsubsection*{\textbf{\emph{Claim 1.}}} For any $N\in\mathbb{N}$, there exists $n\geq N$ such that $s_{n,B}(\mathcal{A})\geq\bar{s}$.
	
	In fact, if there exists $N_0\in\mathbb{N}$, such that for all $n\geq N_0$, $s_{n,B}(\mathcal{A})<\bar{s}$, then \begin{equation*}
		\max\{s_{N_0,B}(\mathcal{A}),s_{N_0+1,B}(\mathcal{A}),\ldots,s_{2N_0-1,B}(\mathcal{A})\}<\bar{s}.
	\end{equation*}
	For any $n\geq2N_0$, write $n=kN_0+r$, $0\leq r\leq N_0-1$, $k\geq2$. Then we have\begin{equation*}
		s_{n,B}(\mathcal{A})=s_{(k-1)N_0+(N_0+r)}(\mathcal{A})\leq\max\{s_{N_0,B}(\mathcal{A}),s_{N_0+r,B}(\mathcal{A})\}<\bar{s}.
	\end{equation*}
	This gives $\limsup_{n\rightarrow\infty}s_{n,B}(\mathcal{A})<\bar{s}$, which follows a contradiction.
	
	For any $n\geq1$ and $k\geq1$, by Lemma \ref{lem2}, since $q_{2(n+k)}\leq q_{2n}q_{2k}$, we have $f_{n+k,B}(\rho)\geq\frac{1}{4}f_{n,B}(\rho)f_{k,B}(\rho)$ for any $0\leq\rho\leq1$. Thus, for any $r\geq1$,
	\begin{equation*}
		f_{rn,B}(\rho)\geq\left(\frac{1}{2}\right)^{2(r-1)}f_{n,B}(\rho)^r.
	\end{equation*}
	For any $n\geq1$ and $r\geq1$, we have $f_{n,B}(s_{rn,B}(\mathcal{A}))^r\leq2^{2(r-1)}f_{rn,B}(s_{rn,B}(\mathcal{A}))<4^{r}$. Then, we immediately have 
	\begin{equation*}
		f_{n,B}(s_{rn,B}(\mathcal{A}))\leq4.
	\end{equation*}
	
	\subsubsection*{\textbf{\emph{Claim 2.}}} For any $r\in\mathbb{N}$, $s_{r,B}(\mathcal{A})\geq\bar{s}$.
	
	By contrary, we suppose that there exists $r_0$ such that $s_{r_0,B}(\mathcal{A})<\bar{s}$. Write $\varepsilon=\bar{s}-s_{r_0,B}$. By Claim 1, we can choose $p$ large enough such that $s_{p,B}(\mathcal{A})\geq\bar{s}$ and $(B^{2p}2^{2p})^\varepsilon>4$. Then, we have
	\begin{align*}
		f_{p,B}(s_{r_0p,B}(\mathcal{A})+\varepsilon)=&\sum_{a_2,\ldots,a_{2p}\in\mathcal{A}}\frac{1}{(B^{2p}q_{2p}^2)^{s_{r_0p,B}(\mathcal{A})+\varepsilon}}	\\
		\leq&\frac{1}{(B^{2p}q_{2p}^2)^\varepsilon}\cdot f_{p,B}(s_{r_0p,B}(\mathcal{A}))<1.
	\end{align*}
	This implies that $s_{p,B}(\mathcal{A})\leq s_{r_0p,B}(\mathcal{A})+\varepsilon\leq s_{r_0,B}(\mathcal{A})+\varepsilon=\bar{s}$, which follows a contradiction. Thus for any $n\in\mathbb{N}$, $s_{n,B}(\mathcal{A})\geq\bar{s}$.
\end{proof}

For any $\alpha\in\mathbb{N}$, take $\mathcal{A_\alpha}=\{1,2,\ldots,\alpha\}$. For simplicity, we write $s_{n,B}(\alpha)$ for $s_{n,B}(\mathcal{A_\alpha})$, $s_B(\alpha)$ for $s_B(\mathcal{A_\alpha})$, $s_{n,B}$ for $s_{n,B}(\mathbb{N})$ and $s_B$ for $s_B(\mathbb{N})$.

\begin{lem}\label{editlem}
	We have
	\begin{equation*}
		\lim_{\alpha\rightarrow\infty}s_B(\alpha)=s_B.
	\end{equation*}
\end{lem}

\begin{proof}
	
	For any $\varepsilon>0$, choose $p_0\in\mathbb{N}$ large enough. Then, for any $p\geq p_0$, $r\geq 1$, we have \begin{equation*}
		f_{p,B}(s_{rp,B}(\mathcal{A})+\varepsilon)<1.
	\end{equation*} This implies that $s_{rp,B}(\mathcal{A})-\varepsilon\leq s_{p,B}(\mathcal{A})-\varepsilon\leq s_{rp,B}(\mathcal{A})$, i.e., $\abs{s_{p,B}(\mathcal{A})-s_{rp,B}(\mathcal{A})}\leq\varepsilon$. Taking $r\rightarrow\infty$, we have \begin{equation*}
		\abs{s_{p,B}(\alpha)-s_{B}(\alpha)}\leq\varepsilon\quad\text{and}\quad\abs{s_{p,B}-s_{B}}\leq\varepsilon.
	\end{equation*}
	Notice that $s_{p_0,B}(\alpha)$ is increasing with respect to $\alpha$. Then $\lim_{\alpha\rightarrow\infty}s_{p_0,B}(\alpha)=s_{p_0,B}$. Thus, there exists $\alpha_0$ such that for any $\alpha\geq\alpha_0$, $\abs{s_{p_0,B}(\alpha)-s_{p_0,B}}\leq\varepsilon$. Then, for any $\alpha\geq\alpha_0$,\begin{equation*}
		\abs{s_{B}(\alpha)-s_{B}}\leq3\varepsilon.
	\end{equation*} Since $\varepsilon$ is arbitrary, this finishes our proof.
\end{proof}

\begin{prop}\label{lem4}
	We have $\lim_{B\rightarrow\infty}s_B=1/2$. Furthermore, $s_B$ is continuous with respect to $B\in(1,\infty)$, and the right limit of $s_B$ at $B=1$ exists, denoted by $s_1$.
\end{prop}

\begin{proof}
	By definition, for all $n\geq1$, and $(a_1,a_2,\ldots,a_n)\in\mathbb{N}^n$, we have
	\begin{equation*}
		\prod_{k=1}^n a_k\leq q_n=q_n(a_1,a_2,\ldots,a_n)\leq\prod_{k=1}^n(a_k+1)\quad\text{and}\quad q_n\geq 2^{\frac{n}{2}-1}.
	\end{equation*}
	We first prove $\lim_{B\rightarrow\infty}s_B=1/2$. On the one hand, we notice that
	\begin{equation*}
		\sum_{a_2,\ldots,a_{2n}}\frac{1}{(B^{2n}q_{2n}^2)^{\frac{1}{2}}}\geq\sum_{a_2,\ldots,a_{2n}}\frac{1}{B^n 2^n \prod_{j=1}^{n}(a_{2j}+1)}\geq\left(\frac{1}{2B}\sum_{k=1}^{\infty}\frac{1}{k+1}\right)^n=\infty.
	\end{equation*}
	On the other hand, for an arbitrary $\varepsilon>0$, we take $B$ large enough such that $B^{1+2\varepsilon}>\sum_{k=1}^{\infty}\frac{1}{k^{1+2\varepsilon}}$ such that
	\begin{equation*}
		\sum_{a_2,\ldots,a_{2n}}\frac{1}{(B^{2n}q_{2n}^2)^{\frac{1}{2}+\varepsilon}}\leq\sum_{a_2,\ldots,a_{2n}}\frac{1}{(B^{2n}\prod_{j=1}^n a_{2j}^2)^{\frac{1}{2}+\varepsilon}}\leq\left(\frac{\sum_{k=1}^{\infty}\frac{1}{k^{1+2\varepsilon}}}{B^{1+2\varepsilon}}\right)^{n}<1.
	\end{equation*}
	Thus $s_{n,B}=1/2$ for all $n\geq1$, which proves $\lim_{B\rightarrow\infty}s_B=1/2$.
	
	Next we prove $s_B$ is continuous with respect to $B\in(1,\infty)$. It suffices to show that $s_{n,B}$ is uniformly continuous with respect to $B$. Notice that $s_{n,B}$ is decreasing with respect to $B$. For any $\varepsilon>0$, when $1<B_0<B<B_0^{1+\varepsilon}$,
	\begin{align*}
		\sum_{a_2,\ldots,a_{2n}}\frac{1}{(B_0^{2n}q_{2n}^2)^{s_{n,B}+\varepsilon}}&\leq\frac{1}{B_0^{2n\varepsilon}}\sum_{a_2,\ldots,a_{2n}}\frac{1}{(B_0^{2n}q_{2n}^2)^{s_{n,B}}} \\
		&\leq\frac{1}{B_0^{2n\varepsilon}}\left(\frac{B}{B_0}\right)^{2ns_{n,B}}\leq\left(\frac{B}{B_0^{1+\varepsilon}}\right)^{2n}<1.
	\end{align*}
	This gives $s_{n,B_0}\leq s_{n,B}+\varepsilon$,
	which shows $s_{n,B}$ is continuous with respect to $B>1$. Obviously the right limit of $s_B$ at $B=1$ exists.
\end{proof}

\section{Proof of Theorem \ref{main1}}\label{3}

We will give a constructive proof by induction. At first, we consider a fixed  real number $x\in[3/2,2)$, for which we present a procedure that gives the partial quotients $a_2,b_2,a_4,b_4,a_6,b_6,\ldots$ successively such that $x=[0;1,a_2,1,a_4,1,a_6,\ldots]+[0;1,b_2,1,b_4,1,b_6,\ldots].$

\begin{construction}[Construction of decomposition in Theorem \ref{main1}]\label{construction1}
	Let $x\in[3/2,2)$. For every $k\in\mathbb{N}$, we choose $a_{2k}$ such that \begin{equation}\label{con1}
		\begin{split}
			[0;1,a_2,1,a_4,\ldots,a_{2k-2},1,a_{2k}]&\leq x-[0;1,b_2,1,b_4,\ldots,b_{2k-2},1]	\\
			&<[0;1,a_2,1,a_4,\ldots,a_{2k-2},1,a_{2k},1],
		\end{split}
	\end{equation}
	and then choose $b_{2k}$ such that \begin{equation}\label{con2}
		\begin{split}
			[0;1,b_2,1,b_4,\ldots,b_{2k-2},1,b_{2k}]&\leq x-[0;1,a_2,1,a_4,\ldots,a_{2k-2},1,a_{2k},1]	\\
			&<[0;1,b_2,1,b_4,\ldots,b_{2k-2},1,b_{2k},1].
		\end{split}
	\end{equation}
\end{construction}

\begin{rem}
	Construction \ref{construction1} gives the decomposition of $x$ as a sum of two continued fractions with infinite expansions. Specially, one might meet the situation that for a certain integer $k$ the inequality ``$\leq$'' in (\ref{con1}) or (\ref{con2}) becomes equality ``$=$'', i.e., if there is an $a_{2k}$ such that $$
	[0;1,a_2,1,a_4,\ldots,a_{2k-2},1,a_{2k}]= x-[0;1,b_2,1,b_4,\ldots,b_{2k-2},1]
	$$ or there is a $b_{2k}$ such that $$
	[0;1,b_2,1,b_4,\ldots,b_{2k-2},1,b_{2k}]= x-[0;1,a_2,1,a_4,\ldots,a_{2k},1].
	$$ Then we can terminate our procedure at that step, obtaining a finite decomposition $$
	x=[0;1,a_2,1,a_4,\ldots,a_{2k-2},1,a_{2k}]+[0;1,b_2,1,b_4,\ldots,b_{2k-2},1]
	$$ or $$
	x=[0;1,b_2,1,b_4,\ldots,b_{2k-2},1,b_{2k}]+[0;1,a_2,1,a_4,\ldots,a_{2k},1].
	$$ Specially, $3/2=[0;1,1]+[0;1]$ and $2=[0;1]+[0;1]$.
\end{rem}

\begin{lem}\label{claim1}
	Let $x\in\left[3/2,2\right)$. \textnormal{(i)} Assume that we have obtained a series of partial quotients $a_2,b_2,a_4,b_4,\ldots,a_{2k-2},b_{2k-2}$ by Construction \ref{construction1}. Then, there exists an $a_{2k}$ satisfying condition (\ref{con1}) if and only if \begin{equation}\label{cl1.1}
		[0;1,a_2,1,a_4,\ldots,a_{2k-2},1,1]\leq x-[0;1,b_2,1,b_4,\ldots,b_{2k-2},1].
	\end{equation} \textnormal{(ii)} Assume that we have obtained $a_2,b_2,a_4,b_4,\ldots,a_{2k-2},b_{2k-2},a_{2k}$ by Construction \ref{construction1}. Then, there exists a $b_{2k}$ satisfying condition (\ref{con2}) if and only if \begin{equation}\label{cl1.2}
		[0;1,b_2,1,b_4,\ldots,b_{2k-2},1,1]\leq x-[0;1,a_2,1,a_4,\ldots,a_{2k},1].
	\end{equation}
\end{lem}

\begin{proof}
	Condition (\ref{con1}) can be rewritten as \begin{equation}\label{con3}
		\begin{split}
			[0;1,a_2,1,a_4,\ldots,a_{2k-2},1,a_{2k}]&\leq x-[0;1,b_2,1,b_4,\ldots,b_{2k-2},1]	\\
			&<[0;1,a_2,1,a_4,\ldots,a_{2k}+1].
		\end{split}
	\end{equation} Condition (\ref{con2}) can be rewritten as \begin{equation}\label{con4}
		\begin{split}
			[0;1,b_2,1,b_4,\ldots,b_{2k-2},1,b_{2k}]&\leq x-[0;1,a_2,1,a_4,\ldots,a_{2k},1]	\\
			&<[0;1,b_2,1,b_4,\ldots,b_{2k}+1].
		\end{split}
	\end{equation}
	
	We first prove claim \textnormal{(i)} of the lemma. Noticing that $a_{2k}\in\mathbb{N}$, the left inequality induces (\ref{cl1.1}), which shows the sufficiency.
	
	For the necessity, we consider the sequence $\{[0;1,a_2,1,a_4,\ldots,a_{2k-2},1,c]\}_{c=1}^{\infty}$, which is strictly increasing with a supremum $[0;1,a_2,1,a_4,\ldots,a_{2k-2},1]$. Thus the subintervals $$
	\big[[0;1,a_2,1,a_4,\ldots,a_{2k-2},1,c],\ [0;1,a_2,1,a_4,\ldots,a_{2k-2},1,c+1]\big)
	$$ give a partition of $\big[[0;1,a_2,1,a_4,\ldots,a_{2k-2},1,1],\ [0;1,a_2,1,a_4,\ldots,a_{2k-2},1]\big)$, denoted by $I_{2k-2}$, i.e., $$
	I_{2k-2}=\bigcup_{c=1}^\infty	\big[[0;1,a_2,1,a_4,\ldots,a_{2k-2},1,c],\ [0;1,a_2,1,a_4,\ldots,a_{2k-2},1,c+1]\big).
	$$ Since we have found $b_{2k-2}$ satisfying condition (\ref{cl1.1}) and (\ref{con4}), we have $$
	[0;1,a_2,\ldots,a_{2k-2},1,1]\leq x-[0;1,b_2,\ldots,b_{2k-2},1]<[0;1,a_2,\ldots,a_{2k-2},1].
	$$ Thus, we can choose a unique $c\in\mathbb{N}$ which is our desired $a_{2k}$ satisfying condition (\ref{con3}).
	
	For the claim \textnormal{(ii)} of the lemma, we apply the same ideas as above. The sufficiency is trivial. For the necessity, we consider the sequence $\{[0;1,b_2,1,b_4,\ldots,b_{2k-2},1,d]\}_{d=1}^{\infty}$, which is strictly increasing with a supremum $[0;1,b_2,1,b_4,\ldots,b_{2k-2},1]$. The partition \begin{equation*}
		\begin{split}
			J_{2k-2}\coloneqq&\big[[0;1,b_2,1,b_4,\ldots,b_{2k-2},1,1],\ [0;1,b_2,1,b_4,\ldots,b_{2k-2},1]\big)	\\
			=&\bigcup_{d=1}^\infty\big[[0;1,b_2,1,b_4,\ldots,b_{2k-2},1,d],\ [0;1,b_2,1,b_4,\ldots,b_{2k-2},1,d+1]\big)
		\end{split}
	\end{equation*} still holds and we also have $x-[0;1,a_2,1,a_4,\ldots,a_{2k},1]\in J_{2k-2}$. Thus we can find a unique $d\in\mathbb{N}$ which is our desired $b_{2k}$ satisfying condition (\ref{con4}).
\end{proof}

\begin{lem}
	Let $x\in\left[3/2,2\right)$. The first step in Construction \ref{construction1} can be done, i.e., both conditions (\ref{con1}) and (\ref{con2}) are satisfied for $k=1$.
\end{lem}

\begin{proof}
	If $k=1$, condition (\ref{con1}) is equivalent to (\ref{cl1.1}) with the form $$
	[0;1,1]\leq x-[0;1],
	$$ which is $x\geq3/2$. Condition (\ref{con2}) is equivalent to (\ref{cl1.2}) with the form $$
	[0;1,1]\leq x-[0;1,a_2,1],
	$$ where $a_2$ is a fixed number related to $x$. Actually it holds for all positive integer $a_2$ as long as $x\geq3/2$.
\end{proof}

\begin{lem}
	Let $x\in\left[3/2,2\right)$. The numbers $a_2,b_2,a_4,b_4,\ldots$ can be chosen such that condition (\ref{con1}) and (\ref{con2}) are satisfied for every $k\in\mathbb{N}$.
\end{lem}

\begin{proof}
	We will prove the statement by contradiction. We assume that $k+1$ is the smallest value for which the conditions (\ref{con1}) and (\ref{con2}) will be violated. Then, there are two possibilities:\begin{itemize}
		\item[(a)] Condition (\ref{con1}) is violated;
		\item[(b)] Condition (\ref{con1}) is satisfied, while condition (\ref{con2}) is violated.
	\end{itemize}
	\subsubsection*{\textbf{\emph{Case (a).}}} In this case, assume that we can only find $a_2,b_2,a_4,b_4\ldots,a_{2k-2},b_{2k-2}$ at most. Since conditions (\ref{con1}) and (\ref{con2}) are satisfied for $k$ and condition (\ref{con1}) is violated for $k+1$, the values of $a_2,b_2,a_4,b_4\ldots,a_{2k-2},b_{2k-2}$ are chosen such that
	\begin{align}
		[0;1,a_2,\ldots,a_{2k-2}]\leq\  &x-[0;1,b_2,\ldots,b_{2k-4},1]<[0;1,a_2,\ldots,a_{2k-2},1],	\label{a1}\\
		[0;1,b_2,\ldots,b_{2k-2}]\leq\  &x-[0;1,a_2,\ldots,a_{2k-2},1]<[0;1,b_2,\ldots,b_{2k-2},1],	\label{a2}\\
		&x-[0;1,b_2,\ldots,b_{2k-2},1]<[0;1,a_2,\ldots,a_{2k-2},1,1].\label{a3}
	\end{align}
	We denote the convergent of the continued fractions \begin{equation}\label{a4}
		[0;1,a_2,1,a_4,\ldots,a_{2k-2},1,1]\quad\text{and}\quad[0;1,b_2,1,b_4,\ldots,b_{2k-2},1]
	\end{equation}
	as
	\begin{equation*}
		\begin{split}
			[0;a_1,a_2,a_3,a_4,\ldots,a_n]&=\frac{p_n}{q_n},\quad 1\leq n\leq 2k,	\\
			[0;b_1,b_2,b_3,b_4,\ldots,a_n]&=\frac{r_n}{s_n},\quad 1\leq n\leq 2k-1.
		\end{split}
	\end{equation*}
	Respectively, we can rewrite conditions (\ref{a1})-(\ref{a3}) as follows:
	\begin{align*}
		\frac{p_{2k-2}}{q_{2k-2}}+\frac{r_{2k-3}}{s_{2k-3}}\leq\  &x<\frac{p_{2k-1}}{q_{2k-1}}+\frac{r_{2k-3}}{s_{2k-3}},	\\
		\frac{p_{2k-1}}{q_{2k-1}}+\frac{r_{2k-2}}{s_{2k-2}}\leq\  &x<\frac{p_{2k-1}}{q_{2k-1}}+\frac{r_{2k-1}}{s_{2k-1}},	\\
		&x<\frac{p_{2k}}{q_{2k}}+\frac{r_{2k-1}}{s_{2k-1}}.
	\end{align*}
	Combining the inequalities above, we get
	\begin{align*}
		\frac{p_{2k-2}}{q_{2k-2}}+\frac{r_{2k-3}}{s_{2k-3}}\leq x< \frac{p_{2k}}{q_{2k}}+\frac{r_{2k-1}}{s_{2k-1}},	\\
		\frac{p_{2k-1}}{q_{2k-1}}+\frac{r_{2k-2}}{s_{2k-2}}\leq x< \frac{p_{2k}}{q_{2k}}+\frac{r_{2k-1}}{s_{2k-1}}.	\\
	\end{align*}
	Therefore, we have 
	\begin{align*}
		\frac{p_{2k}}{q_{2k}}-\frac{p_{2k-2}}{q_{2k-2}}>\frac{r_{2k-3}}{s_{2k-3}}-\frac{r_{2k-1}}{s_{2k-1}}>0,	\\
		\frac{r_{2k-1}}{s_{2k-1}}-\frac{r_{2k-2}}{s_{2k-2}}>\frac{p_{2k-1}}{q_{2k-1}}-\frac{p_{2k}}{q_{2k}}>0.
	\end{align*}
	Using identities (\ref{pre1}) and (\ref{pre2}), we get 
	\begin{equation}\label{a5}
		\begin{split}
			\frac{a_{2k}}{q_{2k}q_{2k-2}}>\frac{b_{2k-1}}{s_{2k-1}s_{2k-3}},	\\
			\frac{1}{s_{2k-1}s_{2k-2}}>\frac{1}{q_{2k}q_{2k-1}}.
		\end{split}		
	\end{equation}
	By (\ref{a4}) above, we have $a_{2k}=1$ and $b_{2k-1}=1$. Plugging these values into (\ref{a5}), we get
	\begin{align*}
		\frac{1}{q_{2k}q_{2k-2}}>\frac{1}{s_{2k-1}s_{2k-3}},	\\
		\frac{1}{s_{2k-1}s_{2k-2}}>\frac{1}{q_{2k}q_{2k-1}},
	\end{align*} which immediately imply that
	\begin{equation*}
		\frac{q_{2k-1}}{s_{2k-2}}>\frac{s_{2k-1}}{q_{2k}}>\frac{q_{2k-2}}{s_{2k-3}},
	\end{equation*} and then
	\begin{equation}\label{a6}
		\frac{q_{2k-1}}{q_{2k-2}}>\frac{s_{2k-2}}{s_{2k-3}}.
	\end{equation}
	Now by (\ref{pre3}), we rewrite (\ref{a6}) as
	\begin{equation*}
		[a_{2k-1};a_{2k-2},a_{2k-3},\ldots,a_2,a_1]>[b_{2k-2};b_{2k-3},b_{2k-4}\ldots,b_2,b_1].
	\end{equation*}
	Since all the odd-order elements are equal to $1$ due to (\ref{a4}), we have
	\begin{equation*}
		[1;a_{2k-2},1,a_{2k-4},\ldots,a_2,1]>[b_{2k-2};1,b_{2k-4},1,\ldots,b_2,1].
	\end{equation*} It follows that $b_{2k-2}=1,a_{2k-2}=1,\ldots,b_2=1\text{, and }a_2=1$ successively. 
	Then we get a contradiction that the inequality above does not hold, since\begin{equation*}
		[\underset{a_{2k-2}=a_{2k-4}=\cdots=a_2=1}{\underbrace{1;a_{2k-2},1,a_{2k-4},\ldots,a_2,1}}]<[\underset{b_{2k-2}=b_{2k-4}=\cdots=b_2=1}{\underbrace{b_{2k-2};1,b_{2k-4},1,\ldots,b_2,1}}].
	\end{equation*} To sum up, case (a) is impossible.
	\subsubsection*{\textbf{\emph{Case (b).}}} In this case, we apply the same ideas as above. Assume that we can only find $a_2,b_2,a_4,b_4\ldots,a_{2k-2},b_{2k-2},a_{2k}$ at most. Since condition (\ref{con1}) is satisfied until step $k+1$ and condition (\ref{con2}) is satisfied for at most $k$, the values of $a_2,b_2,a_4,b_4\ldots,a_{2k-2},b_{2k-2},a_{2k}$ are chosen such that
	\begin{align}
		[0;1,a_2,\ldots,a_{2k}]\leq\  &x-[0;1,b_2,\ldots,b_{2k-2},1]<[0;1,a_2,\ldots,a_{2k},1],	\label{b1}\\
		[0;1,b_2,\ldots,b_{2k-2}]\leq\  &x-[0;1,a_2,\ldots,a_{2k-2},1]<[0;1,b_2,\ldots,b_{2k-2},1],	\label{b2}\\
		&x-[0;1,a_2,\ldots,a_{2k},1]<[0;1,b_2,\ldots,b_{2k-2},1,1].\label{b3}
	\end{align}
	Following the same notation for the continued fractions 
	\begin{equation}\label{b4}
		[0;1,a_2,1,a_4,\ldots,a_{2k},1]\quad\text{and}\quad[0;1,b_2,1,b_4,\ldots,b_{2k-2},1,1],
	\end{equation} we can rewrite conditions (\ref{b1})-(\ref{b3}) as follows:
	\begin{align*}
		\frac{p_{2k}}{q_{2k}}+\frac{r_{2k-1}}{s_{2k-1}}\leq\  &x<\frac{p_{2k+1}}{q_{2k+1}}+\frac{r_{2k-1}}{s_{2k-1}},	\\
		\frac{p_{2k-1}}{q_{2k-1}}+\frac{r_{2k-2}}{s_{2k-2}}\leq\  &x<\frac{p_{2k-1}}{q_{2k-1}}+\frac{r_{2k-1}}{s_{2k-1}},	\\
		&x<\frac{p_{2k+1}}{q_{2k+1}}+\frac{r_{2k}}{s_{2k}}.
	\end{align*}
	Combining the inequalities above, we get
	\begin{align*}
		\frac{p_{2k}}{q_{2k}}+\frac{r_{2k-1}}{s_{2k-1}}\leq x< \frac{p_{2k+1}}{q_{2k+1}}+\frac{r_{2k}}{s_{2k}},	\\
		\frac{p_{2k-1}}{q_{2k-1}}+\frac{r_{2k-2}}{s_{2k-2}}\leq x< \frac{p_{2k+1}}{q_{2k+1}}+\frac{r_{2k}}{s_{2k}}.	\\
	\end{align*}
	Therefore, we have 
	\begin{align*}
		\frac{p_{2k+1}}{q_{2k+1}}-\frac{p_{2k}}{q_{2k}}>\ &\frac{r_{2k-1}}{s_{2k-1}}-\frac{r_{2k}}{s_{2k}}>0,	\\
		\frac{r_{2k}}{s_{2k}}-\frac{r_{2k-2}}{s_{2k-2}}>\ &\frac{p_{2k-1}}{q_{2k-1}}-\frac{p_{2k+1}}{q_{2k+1}}>0.
	\end{align*}
	Using identities (\ref{pre1}) and (\ref{pre2}), we get 
	\begin{equation}\label{b5}
		\begin{split}
			\frac{1}{q_{2k+1}q_{2k}}>\frac{1}{s_{2k-1}s_{2k-3}},	\\
			\frac{b_{2k}}{s_{2k}s_{2k-2}}>\frac{a_{2k+1}}{q_{2k+1}q_{2k-1}}.
		\end{split}		
	\end{equation}
	By (\ref{b4}) above, we have $a_{2k+1}=1$ and $b_{2k}=1$. Plugging these values into (\ref{b5}), we get
	\begin{align*}
		\frac{1}{q_{2k+1}q_{2k}}>\frac{1}{s_{2k-1}s_{2k-3}},	\\
		\frac{1}{s_{2k}s_{2k-2}}>\frac{1}{q_{2k+1}q_{2k-1}},
	\end{align*} which immediately imply that
	\begin{equation*}
		\frac{q_{2k-1}}{s_{2k-2}}>\frac{s_{2k}}{q_{2k+1}}>\frac{q_{2k}}{s_{2k-1}},
	\end{equation*} and then
	\begin{equation}\label{b6}
		\frac{s_{2k-1}}{s_{2k-2}}>\frac{q_{2k}}{q_{2k-1}}.
	\end{equation}
	Now by (\ref{pre3}), we rewrite (\ref{b6}) as
	\begin{equation*}
		[b_{2k-1};b_{2k-2},b_{2k-3},\ldots,b_2,b_1]>[a_{2k};a_{2k-1},a_{2k-2}\ldots,a_2,a_1].
	\end{equation*} Since all the odd-order elements are equal to $1$ due to (\ref{b4}), we have 
	\begin{equation*}
		[1;b_{2k-2},1,\ldots,b_2,1]>[a_{2k};1,a_{2k-2}\ldots,a_2,1],
	\end{equation*} which follows $a_{2k}=1,b_{2k-2}=1,a_{2k-2}=1,\ldots,b_2=1\text{, and }a_2=1$ successively. Then we get the contradiction that the inequality above does not hold, since 
	\begin{equation*}
		[\underset{b_{2k-2}=b_{2k-4}=\cdots=b_2=1}{\underbrace{1;b_{2k-2},1,\ldots,b_2,1}}]<[\underset{a_{2k}=a_{2k-2}=\cdots=a_2=1}{\underbrace{a_{2k};1,a_{2k-2}\ldots,a_2,1}}].
	\end{equation*} To sum up, case (b) is impossible as well.
\end{proof}

\begin{prop}\label{sum}
	Every real number $x\in[3/2,2)$ can be written as a sum of two continued fractions whose odd-order partial quotients are equal to $1$, i.e.,\begin{equation}
		x=[0;1,a_2,1,a_4,\ldots]+[0;1,b_2,1,b_4,\ldots].
	\end{equation}
\end{prop}

\begin{proof}
	Suppose we have chosen values $(a_{2k})_{k=1}^\infty$ and $(b_{2k})_{k=1}^\infty$ such that for every $k\in\mathbb{N}$, condition (\ref{con1}) holds true, i.e., \begin{equation}\label{s1}
		\begin{split}
			[0;1,a_2,1,a_4,\ldots,a_{2k-2},1,a_{2k}]+[0;1,b_2,1,b_4,\ldots,b_{2k-2},1]\leq x	\\
			<[0;1,a_2,1,a_4,\ldots,a_{2k-2},1,a_{2k},1]+[0;1,b_2,1,b_4,\ldots,b_{2k-2},1].
		\end{split}
	\end{equation} Each of the two continued fractions
	\begin{equation*}
		\alpha=[0;1,a_2,1,a_4,\ldots],	\quad\beta=[0;1,b_2,1,b_4,\ldots],
	\end{equation*} has a finite limit. Consequently, inequality (\ref{s1}) implies 
	\begin{equation*}
		\begin{split}
			\lim_{k\rightarrow\infty}[0;1,a_2,1,a_4,\ldots,a_{2k-2},1,a_{2k}]+[0;1,b_2,1,b_4,\ldots,b_{2k-2},1]\leq x	\\
			\leq\lim_{k\rightarrow\infty}[0;1,a_2,1,a_4,\ldots,a_{2k-2},1,a_{2k},1]+[0;1,b_2,1,b_4,\ldots,b_{2k-2},1],
		\end{split}
	\end{equation*} and so we have $\alpha+\beta\leq x\leq\alpha+\beta$. Hence
	\begin{equation*}
		x=\alpha+\beta=[0;1,a_2,1,a_4,\ldots]+[0;1,b_2,1,b_4,\ldots],
	\end{equation*} as we need.
\end{proof}

As a corollary, we prove our main Theorem \ref{main1}.

\begin{proof}[Proof of Theorem \ref{main1}]
	For every real number $x\in\mathbb{Z}+[1/2,1]$, we can find a unique $m\in\mathbb{Z}$ such that $x-m\in[3/2,2]$. If $x-m=2$, write $x=[m;1]+[0;1]$. Otherwise, for $x-m\in[3/2,2)$, by Proposition \ref{sum}, we can find two continued fractions $\alpha$ and $\beta$ such that 
	\begin{equation*}
		x-m=\alpha+\beta=[0;1,a_2,1,a_4,\ldots]+[0;1,b_2,1,b_4,\ldots].
	\end{equation*} Write 
	\begin{equation*}
		x=[m;1,a_2,1,a_4,\ldots]+[0;1,b_2,1,b_4,\ldots],
	\end{equation*} which is our desired result.
\end{proof}

\section{Proof of Theorem \ref{main2}}\label{4}

Similarly as in the proof of Theorem \ref{main1}, we will give a constructive proof by induction. We start from a positive real number $x\in\mathbb{Z}+[1/2,1]$, for which we present a procedure that gives the partial quotients $a_0,b_0,a_2,b_2,a_4,b_4,\ldots$ successively such that $x=[a_0;1,a_2,1,a_4,1,a_6,\ldots]\cdot[b_0;1,b_2,1,b_4,1,b_6,\ldots]$.

\begin{construction}[Construction of decomposition in Theorem \ref{main2}]\label{construction2}
	Let $x\in\mathbb{Z}+[1/2,1]$ be a positive real number. We set \begin{equation*}
		a_0=[x],\quad b_0=0.
	\end{equation*}
	For every $k\in\mathbb{N}$, we choose $a_{2k}$ such that
	\begin{equation}\label{Con1}
		\begin{split}
			[a_0;1,a_2,1,a_4,\ldots,a_{2k-2},1,a_{2k}]&\leq \frac{x}{[b_0;1,b_2,1,b_4,\ldots,b_{2k-2},1]}	\\
			&<[a_0;1,a_2,1,a_4,\ldots,a_{2k-2},1,a_{2k},1],
		\end{split}
	\end{equation}
	and then choose $b_{2k}$ such that \begin{equation}\label{Con2}
		\begin{split}
			[b_0;1,b_2,1,b_4,\ldots,b_{2k-2},1,b_{2k}]&\leq \frac{x}{[a_0;1,a_2,1,a_4,\ldots,a_{2k-2},1,a_{2k},1]}	\\
			&<[b_0;1,b_2,1,b_4,\ldots,b_{2k-2},1,b_{2k},1].
		\end{split}
	\end{equation}
\end{construction}

\begin{rem}
	Construction \ref{construction2} gives the decomposition of $x$ as a positive of two continued fractions with infinite expansions. Specially, one might meet the situation that for a certain integer $k$ the inequality ``$\leq$'' in (\ref{Con1}) or (\ref{Con2}) becomes equality ``$=$'', i.e., if there is an $a_{2k}$ such that $$
	x=[a_0;1,a_2,1,a_4,\ldots,a_{2k-2},1,a_{2k}]\cdot[b_0;1,b_2,1,b_4,\ldots,b_{2k-2},1]
	$$ or there is a $b_{2k}$ such that $$
	x=[b_0;1,b_2,1,b_4,\ldots,b_{2k-2},1,b_{2k}]\cdot[a_0;1,a_2,1,a_4,\ldots,a_{2k},1].
	$$ Then, we can terminate our procedure at that step, obtaining a finite decomposition $$
	x=[a_0;1,a_2,1,a_4,\ldots,a_{2k-2},1,a_{2k}]\cdot[b_0;1,b_2,1,b_4,\ldots,b_{2k-2},1]
	$$ or $$
	x=[b_0;1,b_2,1,b_4,\ldots,b_{2k-2},1,b_{2k}]\cdot[a_0;1,a_2,1,a_4,\ldots,a_{2k},1].
	$$ Specially, $1/2=[0;1,1]\cdot[0;1]$ and $1=[0;1]\cdot[0;1]$.
\end{rem}

\begin{lem}\label{claim2}
	Let $x\in\mathbb{Z}+[1/2,1)$ be a positive real number. \textnormal{(i)} Assume that we have obtained $a_0,b_0,a_2,b_2,\ldots,a_{2k-2},b_{2k-2}$ by Construction \ref{construction2}. Then there exists an $a_{2k}$ satisfying condition (\ref{Con1}) if and only if \begin{equation}\label{cl2.1}
		[a_0;1,a_2,1,a_4,\ldots,a_{2k-2},1,1]\leq \frac{x}{[b_0;1,b_2,1,b_4,\ldots,b_{2k-2},1]}.
	\end{equation} \textnormal{(ii)} Assume that we have obtained $a_0,b_0,a_2,b_2,\ldots,a_{2k-2},b_{2k-2},a_{2k}$ by Construction \ref{construction2}. Then there exists a $b_{2k}$ satisfying condition (\ref{Con2}) if and only if \begin{equation}\label{cl2.2}
		[b_0;1,b_2,1,b_4,\ldots,b_{2k-2},1,1]\leq \frac{x}{[a_0;1,a_2,1,a_4,\ldots,a_{2k},1]}.
	\end{equation}
\end{lem}

\begin{proof}
	Condition (\ref{Con1}) can be rewrited as \begin{equation}\label{Con3}
		\begin{split}
			[a_0;1,a_2,1,a_4,\ldots,a_{2k-2},1,a_{2k}]&\leq \frac{x}{[b_0;1,b_2,1,b_4,\ldots,b_{2k-2},1]}	\\
			&<[a_0;1,a_2,1,a_4,\ldots,a_{2k}+1].
		\end{split}
	\end{equation} Condition (\ref{Con2}) can be rewrited as \begin{equation}\label{Con4}
		\begin{split}
			[b_0;1,b_2,1,b_4,\ldots,b_{2k-2},1,b_{2k}]&\leq \frac{x}{[a_0;1,a_2,1,a_4,\ldots,a_{2k},1]}	\\
			&<[b_0;1,b_2,1,b_4,\ldots,b_{2k}+1].
		\end{split}
	\end{equation}
	
	We first prove claim \textnormal{(i)} of the lemma. Since $a_{2k}\in\mathbb{N}$, the left inequality induces (\ref{cl2.1}), which shows the sufficiency.
	
	For the necessity, we consider the sequence $\{[a_0;1,a_2,1,a_4,\ldots,a_{2k-2},1,c]\}_{c=1}^{\infty}$, which is strictly increasing with a supremum $[a_0;1,a_2,1,a_4,\ldots,a_{2k-2},1]$. Thus, the subintervals $$
	\big[[a_0;1,a_2,1,a_4,\ldots,a_{2k-2},1,c],\ [a_0;1,a_2,1,a_4,\ldots,a_{2k-2},1,c+1]\big)
	$$ give a partition of $\big[[a_0;1,a_2,1,a_4,\ldots,a_{2k-2},1,1],\ [a_0;1,a_2,1,a_4,\ldots,a_{2k-2},1]\big)$, denoted by $I_{2k-2}$, i.e., $$
	I_{2k-2}=\bigcup_{c=1}^\infty	\big[[a_0;1,a_2,1,a_4,\ldots,a_{2k-2},1,c],\ [a_0;1,a_2,1,a_4,\ldots,a_{2k-2},1,c+1]\big).
	$$ Since we have found $b_{2k-2}$ satisfying condition (\ref{cl2.1}) and (\ref{Con4}), we have $$
	[a_0;1,a_2,\ldots,a_{2k-2},1,1]\leq \frac{x}{[b_0;1,b_2,\ldots,b_{2k-2},1]}<[a_0;1,a_2,\ldots,a_{2k-2},1].
	$$ Thus, we can choose a unique $c\in\mathbb{N}$ which is our desired $a_{2k}$ satisfying condition (\ref{Con3}).
	
	For claim \textnormal{(ii)} of the lemma, we apply the same ideas as above. The sufficiency of the claim is trivial. For the necessity, we consider the sequence of continued fractions $\{[b_0;1,b_2,1,b_4,\ldots,b_{2k-2},1,d]\}_{d=1}^{\infty}$, which is strictly increasing with a supremum $[b_0;1,b_2,1,b_4,\ldots,b_{2k-2},1]$. The partition \begin{equation*}
		\begin{split}
			J_{2k-2}\coloneqq&\big[[b_0;1,b_2,1,b_4,\ldots,b_{2k-2},1,1],\ [b_0;1,b_2,1,b_4,\ldots,b_{2k-2},1]\big)	\\
			=&\bigcup_{d=1}^\infty\big[[b_0;1,b_2,1,b_4,\ldots,b_{2k-2},1,d],\ [b_0;1,b_2,1,b_4,\ldots,b_{2k-2},1,d+1]\big)
		\end{split}
	\end{equation*} still holds and we also have ${x}/{[a_0;1,a_2,1,a_4,\ldots,a_{2k},1]}\in J_{2k-2}$. Thus we can find a unique $d\in\mathbb{N}$ which is our desired $b_{2k}$ satisfying condition (\ref{Con4}).
\end{proof}

\begin{lem}
	Let $x\in\mathbb{Z}+[1/2,1)$ be a positive real number. The first step in our construction \ref{construction2} can be done, i.e., both condition (\ref{Con1}) and (\ref{Con2}) are satisfied for $k=0$ and $k=1$.
\end{lem}

\begin{proof}
	If $k=0$, condition (\ref{Con1}) is trivial. Condition (\ref{Con2}) is
	\begin{equation}\label{e1}
		[b_0]\leq\frac{x}{[a_0;1]}<[b_0;1].
	\end{equation} Plugging the values $a_0=[x]$ and $b_0=0$ into \ref{e1}, we have $0\leq x/([x]+1)<1$.
	
	If $k=1$, then condition (\ref{Con1}) is equivalent to (\ref{cl2.1}) with the form 
	\begin{equation*}
		[a_0;1,1]\leq\frac{x}{[b_0;1]}.
	\end{equation*} Plugging the values $a_0=[x]$ and $b_0=0$ into it, we have $[x]+1/2\leq x$, which is $x\in\mathbb{Z}+[1/2,1)$. Condition (\ref{Con2}) is equivalent to (\ref{cl2.2}) with the form
	\begin{equation*}
		[b_0;1,1]\leq\frac{x}{[a_0;1,a_2,1]}.
	\end{equation*} Note that $a_2\in\mathbb{N}$ satisfies condition (\ref{Con1}) when $k=1$, i.e.,
	\begin{equation*}
		[a_0;1,a_2]\leq\frac{x}{[b_0;1]}<[a_0;1,a_2,1],
	\end{equation*} and observe that 
	\begin{equation*}
		[a_0;1,a_2]\cdot[b_0;1]-[a_0;1,a_2,1]\cdot[b_0;1,1]=\frac{1}{2}\left(a_0+\frac{a_2^2+2a_2-1}{(a_2+1)(a_2+2)}\right)>0,
	\end{equation*} which is equivalent to 
	\begin{equation*}
		x\geq[a_0;1,a_2]\cdot[b_0;1]>[a_0;1,a_2,1]\cdot[b_0;1,1].
	\end{equation*} Our desired result follows.
\end{proof}

\begin{lem}
	Let $x\in\mathbb{Z}+[1/2,1)$ be a positive real number. The numbers $a_0,b_0,a_2,b_2,\ldots$ can be chosen such that conditions (\ref{Con1}) and (\ref{Con2}) are satisfied for every $k\in\mathbb{N}$.
\end{lem}

\begin{proof}
	We will prove the statement by contradiction. We assume that $k+1$ is the smallest value for which conditions (\ref{Con1}) and (\ref{Con2}) will be violated. Then there are two possibilities:\begin{itemize}
		\item[(a)] Condition (\ref{Con1}) is violated;
		\item[(b)] Condition (\ref{Con1}) is satisfied, while the condition (\ref{Con2}) is violated.
	\end{itemize}
	\subsubsection*{\textbf{\emph{Case (a).}}} In this case, assume that we can only find $a_0,b_0,a_2,b_2,a_4,b_4\ldots,a_{2k-2},b_{2k-2}$ at most. Since conditions (\ref{Con1}) and (\ref{Con2}) are satisfied for $k$ and condition (\ref{Con1}) is violated for $k+1$, the values of $a_0,b_0,a_2,b_2,a_4,b_4\ldots,a_{2k-2},b_{2k-2}$ are chosen such that \begin{align}
		[a_0;1,a_2,\ldots,a_{2k-2}]\leq\  &\frac{x}{[b_0;1,b_2,\ldots,b_{2k-4},1]}<[a_0;1,a_2,\ldots,a_{2k-2},1],	\label{A1}\\
		[b_0;1,b_2,\ldots,b_{2k-2}]\leq\  &\frac{x}{[a_0;1,a_2,\ldots,a_{2k-2},1]}<[b_0;1,b_2,\ldots,b_{2k-2},1],	\label{A2}\\
		&\frac{x}{[b_0;1,b_2,\ldots,b_{2k-2},1]}<[a_0;1,a_2,\ldots,a_{2k-2},1,1].\label{A3}
	\end{align}
	We denote the convergents of the continued fractions \begin{equation}\label{A4}
		[a_0;1,a_2,1,a_4,\ldots,a_{2k-2},1,1]\quad\text{and}\quad[b_0;1,b_2,1,b_4,\ldots,b_{2k-2},1]
	\end{equation}
	as
	\begin{equation*}
		\begin{split}
			[a_0;a_1,a_2,a_3,a_4,\ldots,a_n]&=\frac{p_n}{q_n},\quad 1\leq n\leq 2k,	\\
			[b_0;b_1,b_2,b_3,b_4,\ldots,a_n]&=\frac{r_n}{s_n},\quad 1\leq n\leq 2k-1.
		\end{split}
	\end{equation*}
	Respectively, we can rewrite conditions (\ref{A1})-(\ref{A3}) as follows:
	\begin{align*}
		\frac{p_{2k-2}}{q_{2k-2}}\cdot\frac{r_{2k-3}}{s_{2k-3}}\leq\  &x<\frac{p_{2k-1}}{q_{2k-1}}\cdot\frac{r_{2k-3}}{s_{2k-3}},	\\
		\frac{p_{2k-1}}{q_{2k-1}}\cdot\frac{r_{2k-2}}{s_{2k-2}}\leq\  &x<\frac{p_{2k-1}}{q_{2k-1}}\cdot\frac{r_{2k-1}}{s_{2k-1}},	\\
		&x<\frac{p_{2k}}{q_{2k}}\cdot\frac{r_{2k-1}}{s_{2k-1}}.
	\end{align*}
	Combining the inequalities above, we get
	\begin{equation*}
		\begin{split}
			\frac{p_{2k-2}}{q_{2k-2}}\cdot\frac{r_{2k-3}}{s_{2k-3}}\leq x< \frac{p_{2k}}{q_{2k}}\cdot\frac{r_{2k-1}}{s_{2k-1}},	\\
			\frac{p_{2k-1}}{q_{2k-1}}\cdot\frac{r_{2k-2}}{s_{2k-2}}\leq x< \frac{p_{2k}}{q_{2k}}\cdot\frac{r_{2k-1}}{s_{2k-1}}.
		\end{split}
	\end{equation*}
	Therefore, we have 
	\begin{equation}\label{A5}
		\begin{split}
			\frac{p_{2k}}{q_{2k}}\cdot\frac{q_{2k-2}}{p_{2k-2}}&>\frac{r_{2k-3}}{s_{2k-3}}\cdot\frac{s_{2k-1}}{r_{2k-1}},	\\
			\frac{p_{2k}}{q_{2k}}\cdot\frac{q_{2k-1}}{p_{2k-1}}&>\frac{r_{2k-2}}{s_{2k-2}}\cdot\frac{s_{2k-1}}{r_{2k-1}}.
		\end{split}
	\end{equation}
	Using identities (\ref{pre1}) and (\ref{pre2}), we have
	\begin{equation}\label{A6}
		\begin{split}
			\frac{p_{2k-1}}{q_{2k-1}}&=\frac{p_{2k}}{q_{2k}}+\frac{1}{q_{2k}q_{2k-1}},	\\
			\frac{r_{2k-1}}{s_{2k-1}}&=\frac{r_{2k-2}}{s_{2k-2}}+\frac{1}{s_{2k-1}s_{2k-2}},	\\
			\frac{p_{2k}}{q_{2k}}&=\frac{p_{2k-2}}{q_{2k-2}}+\frac{a_{2k}}{q_{2k}q_{2k-2}},	\\
			\frac{r_{2k-3}}{s_{2k-3}}&=\frac{r_{2k-1}}{s_{2k-1}}+\frac{b_{2k-1}}{s_{2k-1}s_{2k-3}}.\\
		\end{split}
	\end{equation} By (\ref{A4}) above, we have $a_{2k}=1$ and $b_{2k-1}=1$. Plugging these values into (\ref{A6}), we get
	\begin{equation*}
		\begin{split}
			\frac{p_{2k-1}}{q_{2k-1}}&=\frac{p_{2k}}{q_{2k}}+\frac{1}{q_{2k}q_{2k-1}},	\\
			\frac{r_{2k-1}}{s_{2k-1}}&=\frac{r_{2k-2}}{s_{2k-2}}+\frac{1}{s_{2k-1}s_{2k-2}},	\\
			\frac{p_{2k}}{q_{2k}}&=\frac{p_{2k-2}}{q_{2k-2}}+\frac{1}{q_{2k}q_{2k-2}},	\\
			\frac{r_{2k-3}}{s_{2k-3}}&=\frac{r_{2k-1}}{s_{2k-1}}+\frac{1}{s_{2k-1}s_{2k-3}}.\\
		\end{split}
	\end{equation*} Then we transform (\ref{A5}) into 
	\begin{equation*}
		\begin{split}
			\left(\frac{p_{2k}}{q_{2k}}-\frac{1}{q_{2k}q_{2k-2}}\right)\frac{q_{2k}}{p_{2k}}<\left(\frac{r_{2k-3}}{s_{2k-3}}-\frac{1}{s_{2k-1}s_{2k-3}}\right)\frac{s_{2k-3}}{r_{2k-3}},	\\
			\left(\frac{p_{2k}}{q_{2k}}+\frac{1}{q_{2k}q_{2k-1}}\right)\frac{q_{2k}}{p_{2k}}<\left(\frac{r_{2k-2}}{s_{2k-2}}+\frac{1}{s_{2k-1}s_{2k-2}}\right)\frac{s_{2k-2}}{r_{2k-2}}.	\\
		\end{split}
	\end{equation*} We simplify and get 
	\begin{equation*}
		\begin{split}
			1-\frac{1}{p_{2k}q_{2k-2}}<1-\frac{1}{s_{2k-1}r_{2k-3}},	\\
			1+\frac{1}{p_{2k}q_{2k-1}}<1+\frac{1}{s_{2k-1}r_{2k-2}},	\\
		\end{split}
	\end{equation*} which immediately imply 
	\begin{equation*}
		\begin{split}
			p_{2k}q_{2k-2}<s_{2k-1}r_{2k-3},	\\
			p_{2k}q_{2k-1}>s_{2k-1}r_{2k-2},	\\
		\end{split}
	\end{equation*} and then
	\begin{equation}\label{A7}
		\frac{q_{2k-1}}{q_{2k-2}}>\frac{r_{2k-2}}{r_{2k-3}}.
	\end{equation}
	Now by identities (\ref{pre3}) and (\ref{pre4}), we rewrite (\ref{A7}) as
	\begin{equation*}
		[a_{2k-1};a_{2k-2},a_{2k-3},\ldots,a_4,a_3,a_2,a_1]>[b_{2k-2};b_{2k-3},b_{2k-4},\ldots,b_3,b_2].
	\end{equation*}
	Since all the odd-order elements are equal to $1$ due to (\ref{A4}), we have
	\begin{equation*}
		[1;a_{2k-2},1,a_{2k-4}\ldots,a_4,1,a_2,1]>[b_{2k-2};1,b_{2k-4},1,\ldots,1,b_2].
	\end{equation*} It follows that $b_{2k-2}=1,a_{2k-2}=1,\ldots,b_4=1,a_4=1$ and $b_2=1$ successively. Note that $a_2$ is a pre-determined constant. Plugging these values into (\ref{A3}), we get
	\begin{equation*}
		\begin{split}
			x&<[\underset{b_0=0,b_2=b_4=\cdots=b_{2k-2}=1}{\underbrace{b_0;1,b_2,1,b_4,\ldots,b_{2k-2},1}}]\cdot[\underset{a_0=[x],a_4=\cdots=a_{2k-2}=1}{\underbrace{a_0;1,a_2,1,a_4,\ldots,a_{2k-2},1,1}}]	\\
			&=[0;\underset{2k-1\text{ times}}{\underbrace{1,1,\ldots,1}}]\cdot[a_0;1,a_2,\underset{2k-2\text{ times}}{\underbrace{1,1,\ldots,1}}].
		\end{split}
	\end{equation*}
	To simplify our calculation, we use the notation $\iota_n=[0;\underset{n\text{ times}}{\underbrace{1,1,\ldots,1}}]$. Notice that $x$ satisfies
	\begin{equation*}
		a_0+\frac{1}{1+\frac{1}{a_2}}\leq x<\iota_{2k-1}\left(a_0+\frac{1}{1+\frac{1}{a_2+\iota_{2k-2}}}\right).
	\end{equation*} However, we have
	\begin{equation*}
		\begin{split}
			&\ \iota_{2k-1}\left(a_0+\frac{1}{1+\frac{1}{a_2+\iota_{2k-2}}}\right)-\left(a_0+\frac{1}{1+\frac{1}{a_2}}\right)	\\
			=&\ (\iota_{2k-1}-1)a_0+\frac{\iota_{2k-1}(a_2+\iota_{2k-2})}{a_2+\iota_{2k-2}+1}-\frac{a_2}{a_2+1}	\\
			=&\ (\iota_{2k-1}-1)a_0+\frac{(\iota_{2k-1}-1)a_2^2+(\iota_{2k-1}-1)(\iota_{2k-2}+1)a_2+\iota_{2k-1}\iota_{2k-2}}{(a_2+1)(a_2+\iota_{2k-2}+1)}	\\
			\leq&\ (\iota_{2k-1}-1)a_0+\frac{(\iota_{2k-1}-1)+(\iota_{2k-1}-1)(\iota_{2k-2}+1)+\iota_{2k-1}\iota_{2k-2}}{(a_2+1)(a_2+\iota_{2k-2}+1)}	\\
			=&\ (\iota_{2k-1}-1)a_0+\frac{(\iota_{2k-1}-1)+(\iota_{2k-1}-1)(\iota_{2k-2}+1)+\iota_{2k-1}\frac{1-\iota_{2k-1}}{\iota_{2k-1}}}{(a_2+1)(a_2+\iota_{2k-2}+1)}	\\
			=&\ (\iota_{2k-1}-1)a_0+\frac{(\iota_{2k-1}-1)(\iota_{2k-2}+1)}{(a_2+1)(a_2+\iota_{2k-2}+1)}	\\
			\leq&\ 0,
		\end{split}
	\end{equation*} which leads to a contradiction. To sum up, case (a) is impossible.

	
	\subsubsection*{\textbf{\emph{Case (b).}}} In this case, we apply the same ideas as in Case (a). Assume that we can only find $a_0,b_0,a_2,b_2,\ldots,a_{2k-2},b_{2k-2},a_{2k}$ at most. Since condition (\ref{Con1}) is satisfied for at most $k+1$ and condition (\ref{Con2}) is satisfied for at most $k$, the values of $a_0,b_0,a_2,b_2,a_4,b_4\ldots,a_{2k-2},b_{2k-2},a_{2k}$ are chosen such that
	\begin{align}
		[a_0;1,a_2,\ldots,a_{2k}]\leq\  &\frac{x}{[b_0;1,b_2,\ldots,b_{2k-2},1]}<[a_0;1,a_2,\ldots,a_{2k},1],	\label{B1}\\
		[b_0;1,b_2,\ldots,b_{2k-2}]\leq\  &\frac{x}{[a_0;1,a_2,\ldots,a_{2k-2},1]}<[b_0;1,b_2,\ldots,b_{2k-2},1],	\label{B2}\\
		&\frac{x}{[a_0;1,a_2,\ldots,a_{2k},1]}<[b_0;1,b_2,\ldots,b_{2k-2},1,1].	\label{B3}
	\end{align}

	Following the same notation for continued fractions 
	\begin{equation}\label{B4}
		[a_0;1,a_2,1,a_4,\ldots,a_{2k},1]\quad\text{and}\quad[b_0;1,b_2,1,b_4,\ldots,b_{2k-2},1,1],
	\end{equation} we can rewrite conditions (\ref{B1})-(\ref{B3}) as follows:
	\begin{align*}
		\frac{p_{2k}}{q_{2k}}\cdot\frac{r_{2k-1}}{s_{2k-1}}\leq\  &x<\frac{p_{2k+1}}{q_{2k+1}}\cdot\frac{r_{2k-1}}{s_{2k-1}},	\\
		\frac{p_{2k-1}}{q_{2k-1}}\cdot\frac{r_{2k-2}}{s_{2k-2}}\leq\  &x<\frac{p_{2k-1}}{q_{2k-1}}\cdot\frac{r_{2k-1}}{s_{2k-1}},	\\
		&x<\frac{p_{2k+1}}{q_{2k+1}}\cdot\frac{r_{2k}}{s_{2k}}.
	\end{align*}
	Combining the inequalities above, we get
	\begin{align*}
		\frac{p_{2k}}{q_{2k}}\cdot\frac{r_{2k-1}}{s_{2k-1}}\leq x< \frac{p_{2k+1}}{q_{2k+1}}\cdot\frac{r_{2k}}{s_{2k}},	\\
		\frac{p_{2k-1}}{q_{2k-1}}\cdot\frac{r_{2k-2}}{s_{2k-2}}\leq x< \frac{p_{2k+1}}{q_{2k+1}}\cdot\frac{r_{2k}}{s_{2k}}.	\\
	\end{align*}
	Therefore, we have 
	\begin{equation}\label{B5}
		\begin{split}
			\frac{p_{2k+1}}{q_{2k+1}}\cdot\frac{q_{2k}}{p_{2k}}&>\frac{r_{2k-1}}{s_{2k-1}}\cdot\frac{s_{2k}}{r_{2k}},	\\
			\frac{p_{2k+1}}{q_{2k+1}}\cdot\frac{q_{2k-1}}{p_{2k-1}}&>\frac{r_{2k-2}}{s_{2k-2}}\cdot\frac{s_{2k}}{r_{2k}}.	\\
		\end{split}
	\end{equation}
	Using identities (\ref{pre1}) and (\ref{pre2}), we have 
	\begin{equation}\label{B6}
		\begin{split}
			\frac{p_{2k+1}}{q_{2k+1}}&=\frac{p_{2k}}{q_{2k}}+\frac{1}{q_{2k+1}q_{2k}},	\\
			\frac{r_{2k-1}}{s_{2k-1}}&=\frac{r_{2k}}{s_{2k}}+\frac{1}{s_{2k}s_{sk-1}},	\\
			\frac{p_{2k-1}}{q_{2k-1}}&=\frac{p_{2k+1}}{q_{2k+1}}+\frac{a_{2k+1}}{q_{2k+1}q_{2k-1}},	\\
			\frac{r_{2k}}{s_{2k}}&=\frac{r_{2k-2}}{s_{2k-2}}+\frac{b_{2k}}{s_{2k}s_{2k-2}}.	\\
		\end{split}
	\end{equation}
	By (\ref{B4}) above, we have $a_{2k+1}=1$ and $b_{2k}=1$. Plugging these values into (\ref{B6}), we get 
	\begin{align*}
		\frac{p_{2k+1}}{q_{2k+1}}&=\frac{p_{2k}}{q_{2k}}+\frac{1}{q_{2k+1}q_{2k}},	\\
		\frac{r_{2k-1}}{s_{2k-1}}&=\frac{r_{2k}}{s_{2k}}+\frac{1}{s_{2k}s_{sk-1}},	\\
		\frac{p_{2k-1}}{q_{2k-1}}&=\frac{p_{2k+1}}{q_{2k+1}}+\frac{1}{q_{2k+1}q_{2k-1}},	\\
		\frac{r_{2k}}{s_{2k}}&=\frac{r_{2k-2}}{s_{2k-2}}+\frac{1}{s_{2k}s_{2k-2}}.	\\
	\end{align*}
	Then we transform (\ref{B5}) into 
	\begin{align*}
		\left(\frac{p_{2k+1}}{q_{2k+1}}-\frac{1}{q_{2k+1}q_{2k}}\right)\frac{q_{2k+1}}{p_{2k+1}}&<\left(\frac{r_{2k-1}}{s_{2k-1}}-\frac{1}{s_{2k}s_{2k-1}}\right)\frac{s_{2k-1}}{r_{2k-1}},	\\
		\left(\frac{p_{2k+1}}{q_{2k+1}}+\frac{1}{q_{2k+1}q_{2k-1}}\right)\frac{q_{2k+1}}{p_{2k+1}}&<\left(\frac{r_{2k-2}}{s_{2k-2}}+\frac{1}{s_{2k}s_{2k-2}}\right)\frac{s_{2k-2}}{r_{2k-2}}.
	\end{align*}
	We simplify and get 
	\begin{align*}
		1-\frac{1}{p_{2k+1}q_{2k}}<1-\frac{1}{r_{2k-1}s_{2k}},	\\
		1+\frac{1}{p_{2k+1}q_{2k-1}}<1+\frac{1}{r_{2k-2}s_{2k}},
	\end{align*}
	which immediately imply 
	\begin{align*}
		p_{2k+1}q_{2k}<r_{2k-1}s_{2k},	\\
		p_{2k+1}q_{2k-1}>r_{2k-2}s_{2k},	\\
	\end{align*}
	and then
	\begin{equation}\label{B7}
		\frac{q_{2k}}{q_{2k-1}}<\frac{r_{2k-1}}{r_{2k-2}}.
	\end{equation}
	Now by identities (\ref{pre3}) and (\ref{pre4}), we rewrite (\ref{B7}) as 
	\begin{equation*}
		[a_{2k};a_{2k-1},a_{2k-2},\ldots,a_4,a_3,a_2,a_1]<[b_{2k-1};b_{2k-2},b_{2k-3},\ldots,b_3,b_2].
	\end{equation*}
	Since all the odd-order elements are equal to $1$ due to definition (\ref{B4}), we have 
	\begin{equation*}
		[a_{2k};1,a_{2k-2},1,\ldots,a_4,1,a_2,1]<[1;b_{2k-2},1,b_{2k-4},\ldots,1,b_2].
	\end{equation*} It follows that $a_{2k}=1,b_{2k-2}=1,\ldots,b_4=1,a_4=1$ and $b_2=1$ successively. Note that $a_2$ is a pre-determined constant. Plugging these values into (\ref{B3}), we get
	\begin{equation*}
		\begin{split}
			x&<[\underset{b_0=0,b_2=b_4=\cdots=b_{2k-2}=1}{\underbrace{b_0;1,b_2,1,b_4,\ldots,b_{2k-2},1,1}}]\cdot[\underset{a_0=[x],a_4=\cdots=a_{2k}=1}{\underbrace{a_0;1,a_2,1,a_4,\ldots,a_{2k},1}}]	\\
			&=[0;\underset{2k\text{ times}}{\underbrace{1,1,\ldots,1}}]\cdot[a_0;1,a_2,\underset{2k-1\text{ times}}{\underbrace{1,1,\ldots,1}}].
		\end{split}
	\end{equation*}
	Similarly, we notice that $x$ satisfies
	\begin{equation*}
		a_0+\frac{1}{1+\frac{1}{a_2}}\leq x<\iota_{2k}\left(a_0+\frac{1}{1+\frac{1}{a_2+\iota_{2k-1}}}\right).
	\end{equation*} However, we have
	\begin{equation*}
		\begin{split}
			&\ \iota_{2k}\left(a_0+\frac{1}{1+\frac{1}{a_2+\iota_{2k-1}}}\right)-\left(a_0+\frac{1}{1+\frac{1}{a_2}}\right)	\\
			=&\ (\iota_{2k}-1)a_0+\frac{(\iota_{2k}-1)a_2^2+(\iota_{2k}-1)(\iota_{2k-1}+1)a_2+\iota_{2k}\iota_{2k-1}}{(a_2+1)(a_2+\iota_{2k-1}+1)}	\\
			\leq&\ (\iota_{2k}-1)a_0+\frac{(\iota_{2k}-1)+(\iota_{2k}-1)(\iota_{2k-1}+1)+\iota_{2k}\iota_{2k-1}}{(a_2+1)(a_2+\iota_{2k-1}+1)}	\\
			=&\ (\iota_{2k}-1)a_0+\frac{(\iota_{2k}-1)(\iota_{2k-1}+1)}{(a_2+1)(a_2+\iota_{2k-1}+1)}	\\
			\leq&\ 0,
		\end{split}
	\end{equation*} which leads to a contradiction. To sum up, case (b) is impossible as well.
\end{proof}

\begin{prop}\label{product}
	Every positive real number $x\in\mathbb{Z}+[1/2,1)$ can be written as a product of two continued fractions whose odd-order partial quotients are equal to $1$, i.e.,\begin{equation}
		x=[a_0;1,a_2,1,a_4,\ldots]\cdot[b_0;1,b_2,1,b_4,\ldots].
	\end{equation}
\end{prop}

\begin{proof}
	We set $a_0=[x]$ and $b_0=0$ firstly. Suppose we have chosen values $(a_{2k})_{k=0}^\infty$ and $(b_{2k})_{k=0}^\infty$ such that for every $k\in\mathbb{N}$, condition (\ref{Con1}) holds true, i.e., \begin{equation}\label{p1}
		\begin{split}
			[a_0;1,a_2,1,a_4,\ldots,a_{2k-2},1,a_{2k}]\cdot[b_0;1,b_2,1,b_4,\ldots,b_{2k-2},1]\leq x	\\
			<[a_0;1,a_2,1,a_4,\ldots,a_{2k-2},1,a_{2k},1]\cdot[b_0;1,b_2,1,b_4,\ldots,b_{2k-2},1].
		\end{split}
	\end{equation} Each of the two continued fractions
	\begin{equation*}
		\alpha=[a_0;1,a_2,1,a_4,\ldots],	\quad\beta=[b_0;1,b_2,1,b_4,\ldots],
	\end{equation*} has a finite limit. Consequently, inequality (\ref{p1}) implies 
	\begin{equation*}
		\begin{split}
			\lim_{k\rightarrow\infty}[a_0;1,a_2,1,a_4,\ldots,a_{2k-2},1,a_{2k}]\cdot[b_0;1,b_2,1,b_4,\ldots,b_{2k-2},1]\leq x	\\
			\leq\lim_{k\rightarrow\infty}[a_0;1,a_2,1,a_4,\ldots,a_{2k-2},1,a_{2k},1]\cdot[b_0;1,b_2,1,b_4,\ldots,b_{2k-2},1],
		\end{split}
	\end{equation*} and so we have $\alpha\cdot\beta\leq x\leq\alpha\cdot\beta$. Hence,
	\begin{equation*}
		x=\alpha\cdot\beta=[a_0;1,a_2,1,a_4,\ldots]\cdot[b_0;1,b_2,1,b_4,\ldots],
	\end{equation*} as we need.
\end{proof}

As a corollary, we prove our second main Theorem \ref{main2}.

\begin{proof}[Proof of Theorem \ref{main2}]
	For every positive real number $x\in\mathbb{Z}+[1/2,1)$, by Proposition \ref{product}, we can find two continued fractions $\alpha$ and $\beta$ such that 
	\begin{equation*}
		x=\alpha\cdot\beta=[a_0;1,a_2,1,a_4,\ldots]\cdot[b_0;1,b_2,1,b_4,\ldots].
	\end{equation*}
	For positive integer $x$, write $x=[x-1;1]\cdot[0;1]$.
\end{proof}

\section{Hausdorff dimension of $F(B)$}\label{5}

In this section, we shall determine the Hausdorff dimension of set
\begin{equation*}
	F(B)=\{x\in I_{even}\colon a_{2n}(x)\geq B^{2n}\ \text{i.m.}\ n\}	\quad(B>1),
\end{equation*}

and prove our third main theorem \ref{main3}. To this end, we estimate the upper bound and lower bound of the Hausdorff dimension of $F(B)$ separately. The upper bound estimation is rather direct while the lower bound estimation is much more complicated. To give a proper upper bound estimation, we directly provide a sequences of covers on our target set. To give a lower bound estimation, we create a series of indices $\{n_k\}_{k\geq 1}$ to simulate the phenomena that partial quotients grow at an exponential rate. We construct a subset $F_\alpha(B)\subset F(B)$ based on $\{n_k\}_{k\geq 1}$ and use Hausdorff dimension of $F_\alpha(B)$ to approximate that of $F(B)$. Our method are motivated by Wang and Wu \cite{MR2419924}.

Fix $\alpha\in\mathbb{N}$ and a strictly increasing sequence $\varLambda=\{n_1,n_2,\ldots\}$ with $n_1=1$ and $n_k\in\mathbb{N}$ for all $k\geq1$. Let
\begin{align*}
	F_\alpha(B)\coloneqq\{x\in I_{even}\colon &[B^{2n_k}]+1\leq a_{2n_k}(x)\leq 2[B^{2n_k}]\ \text{for all}\ k\geq1,\ \text{and}	\\
	&1\leq a_{2j}(x)\leq\alpha,\ \text{for all}\ j\ne n_k\}.
\end{align*}

We will use a kind of symbolic space defined as follows to describe the structure of $F_\alpha(B)$. For any $n\geq1$, we define
\begin{align*}
	D_n\coloneqq\{(\sigma_2,\ldots,\sigma_{2n})\in\mathbb{N}^n\colon&[B^{2n_k}]+1\leq \sigma_{2n_k}\leq 2[B^{2n_k}],\ 1\leq n_k\leq n,\ \text{and}	\\
	&1\leq\sigma_{2j}\leq\alpha,\ 1\leq j\ne n_k\leq n\},
\end{align*}
\begin{equation*}
	D=\bigcup^\infty_{n=0}D_n\quad(D_0=\varnothing).
\end{equation*}
For any $n\geq1$ and $(\sigma_2,\ldots,\sigma_{2n})\in D_{n}$, we call $I(1,\sigma_2,\ldots,1,\sigma_{2n})$ a \textit{basic interval of order $n$} and 
\begin{equation}\label{c1}
	J(1,\sigma_2,\ldots,1,\sigma_{2n})=\bigcup_{\sigma_{2(n+1)}}\textnormal{cl}I(1,\sigma_2,\ldots,1,\sigma_{2n},1,\sigma_{2(n+1)})
\end{equation}
a \textit{fundamental interval of order $n$}, where the union in (\ref{c1}) is taken over all $\sigma_{2(n+1)}$ such that $(\sigma_2,\ldots,\sigma_{2n},\sigma_{2(n+1)})\in D_{n+1}$. From (\ref{pre5}), if $n\ne n_k-1$ for any $k\geq1$, then
\begin{equation}\label{c2}
	\abs{J(1,\sigma_2,\ldots,1,\sigma_{2n})}=\frac{\alpha}{(q_{2n+1}+q_{2n})((\alpha+1)q_{2n+1}+q_{2n})},
\end{equation} if $n=n_k-1$ for some $k\geq1$, then
\begin{equation}\label{c3}
	\abs{J(1,\sigma_2,\ldots,1,\sigma_{2n})}=\frac{[B^{2n_k}]}{(([B^{2n_k}]+1)q_{2n+1}+q_{2n})((2[B^{2n_k}]+1)q_{2n+1}+q_{2n})}.
\end{equation} It is clearly that \begin{equation}\label{c4}
	F_\alpha(B)=\bigcap_{n\geq1}\bigcup_{(\sigma_2,\ldots,\sigma_{2n})\in D_n}J(1,\sigma_2,\ldots,1,\sigma_{2n}).
\end{equation}

\subsection{Upper bound of Hausdorff dimension of $F(B)$}

We directly calculate the upper bound of $\dim_H F(B)$.

\begin{prop}\label{upper}
	For any $B>1$, we have $\dim_H F(B)\leq s_B$.
\end{prop}

\begin{proof}
	Notice that \begin{align*}
		F(B)&=\bigcap_{N=1}^{\infty}\bigcup_{n\geq N}\{x\in I_{even}\colon a_{2(n+1)}(x)\geq B^{2(n+1)}\}	\\
		&=\bigcap_{N=1}^{\infty}\bigcup_{n\geq N}\bigcup_{a_2,\ldots,a_{2n}}\{x\in I_{even}\colon a_{2k}(x)=a_{2k},\ 1\leq k\leq n,\ a_{2(n+1)}(x)\geq B^{2(n+1)}\}	\\
		&\coloneqq\bigcap_{N=1}^{\infty}\bigcup_{n\geq N}\bigcup_{a_2,\ldots,a_{2n}}J(1,a_2,\ldots,1,a_{2n}).
	\end{align*} From (\ref{pre5}), we have 
	\begin{equation*}
		\abs{J(1,a_2,\ldots,1,a_{2n})}\leq\frac{1}{q_{2n+1}(B^{2(n+1)}q_{2n+1}+q_{2n})}\leq\frac{1}{B^{2n}q_{2n}^2}.
	\end{equation*}
	For any $\varepsilon>0$, we have $s_{n,B}<s_B+\frac{\varepsilon}{2}$ when $n$ is large enough. By the definition of $s_{n,B}$ and the fact that $q_n>2^{n/2-1}$, we have \begin{align*}
		\mathcal{H}^{s_B+\varepsilon}(F(B))&\leq\liminf_{N\rightarrow \infty}\sum_{n\geq N}\sum_{a_2,\ldots,a_{2n}}\abs{J(1,a_2,\ldots,1,a_{2n})}^{s_B+\varepsilon}	\\
		&\leq\liminf_{N\rightarrow \infty}\sum_{n\geq N}\sum_{a_2,\ldots,a_{2n}}\frac{1}{(B^{2n}q_{2n}^2)^{s_B+\varepsilon}}	\\
		&\leq\liminf_{N\rightarrow \infty}\sum_{n\geq N}\sum_{a_2,\ldots,a_{2n}}\frac{1}{(B^{2n}q_{2n}^2)^{s_{n,B}+\varepsilon/2}}	\\
		&\leq\liminf_{N\rightarrow \infty}\sum_{n\geq N}\sum_{a_2,\ldots,a_{2n}}\frac{1}{(B^{2n}q_{2n}^2)^{s_B}}\cdot\frac{1}{B^{n\varepsilon}2^{(n-2)\varepsilon/2}}	\\
		&\leq\liminf_{N\rightarrow \infty}\sum_{n\geq N}\frac{1}{B^{n\varepsilon}2^{(n-2)\varepsilon/2}}=0.
	\end{align*} It follows that $\dim_H F(B)\leq s_B+\varepsilon$. Since $\varepsilon>0$ is arbitrary, we have $\dim_H F(B)\leq s_B$.
\end{proof}

\subsection{Construction of a suitable measure}

In this subsection, we define a probability measure $\mu$ supported on $F_\alpha(B)$ and give the estimation of $\mu(J(1,\sigma_2,\ldots,1,\sigma_{2n}))$, for any $(\sigma_2,\ldots,\sigma_{2n})\in D_n$.

Let $\{n_k\}_{k\geq1}$ be an appropriate sequence of positive integers with $1=n_1<n_2<\cdots<n_k<\cdots$, and define $m_1=n_1-1=0$, $m_k=n_k-n_{k-1}-1$ for any $k\geq2$. Now we define a set function $\mu\colon\{J(1,\sigma_2,\ldots,1,\sigma_{2k})\colon(\sigma_2,\ldots,\sigma_{2k})\in D\backslash D_0,\ k\geq1\}\rightarrow\mathbb{R}^+$, given as follows.

For any $(\sigma_2)\in D_1=D_{n_1}$, let
\begin{equation*}
	\mu(J(1,\sigma_2))=\frac{1}{[B^2]}.
\end{equation*}
For any $(\sigma_2,\ldots,\sigma_{2(n_2-1)})\in D_{n_2-1}$, let
\begin{equation*}
	\mu(J(1,\sigma_2,\ldots,1,\sigma_{2(n_2-1)}))=\mu(J(1,\sigma_2))\left(\frac{1}{B^{2m_2}q_{2m_2}^2(1,\sigma_4,\ldots,1,\sigma_{2(n_2-1)})}\right)^{s_{m_2,B}(\alpha)}
\end{equation*}
and for any $(\sigma_2,\ldots,\sigma_{2n_2})\in D_{n_2}$, let
\begin{equation*}
	\mu(J(1,\sigma_2,\ldots,1,\sigma_{2n_2}))=\frac{1}{[B^{2n_2}]}\mu(J(1,\sigma_2,\ldots,1,\sigma_{2(n_2-1)})).
\end{equation*}
For any $k\geq2$, $\mu(J(1,\sigma_2,\ldots,1,\sigma_{2n_k}))$ and $\mu(J(1,\sigma_2,\ldots,1,\sigma_{2(n_{k+1}-1)}))$ can be defined recursively for any $(\sigma_2,\ldots,\sigma_{2n_k})\in D_{n_k}$ and any $(\sigma_2,\ldots,\sigma_{2(n_{k+1}-1)})\in D_{n_{k+1}-1}$. Let \begin{equation*}
	\mu(J(1,\sigma_2,\ldots,1,\sigma_{2n_k}))=\frac{1}{[B^{2n_k}]}\mu(J(1,\sigma_2,\ldots,1,\sigma_{2(n_{k}-1)})),
\end{equation*}
and let
\begin{align*}
	&\ \mu(J(1,\sigma_2,\ldots,1,\sigma_{2(n_{k+1}-1)}))	\\
	=&\ \mu(J(1,\sigma_2,\ldots,1,\sigma_{2n_k}))\left(\frac{1}{B^{2m_{k+1}}q_{2m_{k+1}}^2(1,\sigma_{2(n_k+1)},\ldots,1,\sigma_{2(n_{k+1}-1)})}\right)^{s_2{m_{k+1},B}(\alpha)}.
\end{align*}
For any $n_{k-1}<n<n_k-1$ and $(\sigma_2,\ldots,\sigma_{2n})\in D_{n}$, let
\begin{equation*}
	\mu(J(1,\sigma_2,\ldots,1,\sigma_{2n}))=\sum_{\sigma_{2(n+1)},\ldots,\sigma_{2(n_k-1)}\leq\alpha}\mu(J(1,\sigma_2,\ldots,1,\sigma_{2n},1,\sigma_{2(n+1)},\ldots,1,\sigma_{2(n_k-1)})).
\end{equation*}

Now, our desired set function $\mu\colon\{J(1,\sigma_2,\ldots,1,\sigma_{2k})\colon(\sigma_2,\ldots,\sigma_{2k})\in D\backslash D_0,\ k\geq1\}\rightarrow\mathbb{R}^+$ is well-defined. By remark \ref{rem1}, it is easy to check that for any $n\geq1$ and $(\sigma_2,\sigma_{2n})\in D_n$, we have
\begin{equation*}
	\mu(J(1,\sigma_2,\ldots,1,\sigma_{2n}))=\sum_{\sigma_{2(n+1)}\colon (\sigma_2,\ldots,\sigma_{2n},\sigma_{2(n+1)})\in D_{n+1}}\mu(J(1,\sigma_2,\ldots,1,\sigma_{2(n+1)})).
\end{equation*}
Notice that 
\begin{equation*}
	\sum_{\sigma_2\in D_1}\mu(J(1,\sigma_2))=1,
\end{equation*}
by Kolmogorov extension theorem, the set function $\mu$ can be extended into a probability measure supported on $F_\alpha(B)$, which is still denoted by $\mu$.

From the definition of $\mu$, we have for $(\sigma_2,\ldots,\sigma_{2n})\in D_n$, if $n=n_k$ for some $k\geq1$, 
\begin{equation}\label{mu1}
	\mu(J(1,\sigma_2,\ldots,1,\sigma_{2n_k}))=\prod_{j=1}^{k}\frac{1}{[B^{2n_j}]}\left(\frac{1}{B^{2m_j}q_{2m_j}^2}\right)^{s_{m_j,B}(\alpha)},
\end{equation}
if $n=n_k-1$ for some $k\geq2$,
\begin{equation}\label{mu2}
	\mu(J(1,\sigma_2,\ldots,1,\sigma_{2(n_k-1)}))=[B^{2n_k}]\prod_{j=1}^{k}\frac{1}{[B^{2n_j}]}\left(\frac{1}{B^{2m_j}q_{2m_j}^2}\right)^{s_{m_j,B}(\alpha)},
\end{equation}
and if $n_{k-1}<n<n_k-1$ for some $k\geq2$,
\begin{equation}\label{mu3}
	\mu(J(1,\sigma_2,\ldots,1,\sigma_{2n}))=\sum_{\sigma_{2(n+1)},\ldots,\sigma_{2(n_k-1)}\leq\alpha}[B^{2n_k}]\prod_{j=1}^{k}\frac{1}{[B^{2n_j}]}\left(\frac{1}{B^{2m_j}q_{2m_j}^2}\right)^{s_{m_j,B}(\alpha)}.
\end{equation}

Now we start to consider the estimation of $\mu(J(1,\sigma_2,\ldots,1,\sigma_{2n}))$, for which $(\sigma_2,\ldots,\sigma_{2n})\in D_n$.

\begin{prop}[Estimation of $\mu(J(1,\sigma_2,\ldots,1,\sigma_{2n}))$]\label{estmu}
	For any $B>1$, $0<\varepsilon<1/4$, take $t=s_B(\alpha)-2\varepsilon>0$. We can choose a proper sequence of $\{n_k\}_{k\geq1}$ and $k_0$ sufficiently large. Then there exists a constant $c_I$ simply related to $\alpha,B,n_1,\ldots,n_{k_0}$ such that for any $n\geq n_{k_0}$ and $(\sigma_2,\ldots,\sigma_{2n})\in D_n$,
	\begin{equation*}
		\mu(J(1,\sigma_2,\ldots,1,\sigma_{2n}))\leq c_I\cdot\abs{J(1,\sigma_2,\ldots,1,\sigma_{2n})}^{t-\varepsilon}.
	\end{equation*}
\end{prop}

\begin{proof}
	
	Suppose we have fixed $\{n_k\}_{k\geq1}$. By our definitions, we assume $k_0$ is sufficiently large such that 
	\begin{align}
		\abs{s_B(\alpha)-s_{m_j,B}(\alpha)}&\leq\varepsilon,\quad j\geq k_0,	\\
		\frac{\log2}{2n_j(t+\varepsilon)\log B}+\frac{t}{t+\varepsilon}&\leq\frac{m_j}{n_j},\quad j\geq k_0.
	\end{align}
	By Remark \ref{rem1}, it is obviously that
	\begin{align*}
		&\ \prod_{j=1}^{k_0}B^{2(n_1+\cdots+n_j)}\alpha^{n_j}\left(\frac{1}{B^{2n_j}q_{2n_j}^2(1,\sigma_2,\ldots,1,\sigma_{2n_j})}\right)^t	\\
		\geq&\ \prod_{j=1}^{k_0}\sum_{(\sigma_{2},\ldots,\sigma_{2n_j})\in D_{n_j}}\left(\frac{1}{B^{2n_j}q_{2n_j}^2(1,\sigma_2,\ldots,1,\sigma_{2n_j})}\right)^t\geq1.
	\end{align*}
	We write $c_I=\prod_{j=1}^{k_0}B^{2(n_1+\cdots+n_j)}\alpha^{n_j}$. By lemma \ref{lem1} and \ref{lem2}, we have
	\begin{equation}\label{q2nk}
		\prod_{j=1}^{k}q_{2m_j}\geq\frac{1}{2^{2k}}\cdot\prod_{j=1}^{k}\frac{1}{\sigma_{2n_j}+1}\cdot q_{2n_k}\geq\frac{1}{2^{2k}}\cdot\prod_{j=1}^{k}\frac{1}{2[B^{2n_j}]+1}\cdot q_{2n_k}.
	\end{equation}
	If $n\ne n_k-1$ for any $k\geq1$, then
	\begin{equation*}
		\abs{J(1,\sigma_2,\ldots,1,\sigma_{2n})}=\frac{\alpha}{(q_{2n+1}+q_{2n})((\alpha+1)q_{2n+1}+q_{2n})}\geq\frac{\alpha}{3(2\alpha+3)q_{2n}^2},
	\end{equation*} If $n=n_k-1$ for some $k\geq1$, then
	\begin{align*}
		\abs{J(1,\sigma_2,\ldots,1,\sigma_{2n})}&=\frac{[B^{2n_k}]}{(([B^{2n_k}]+1)q_{2n+1}+q_{2n})((2[B^{2n_k}]+1)q_{2n+1}+q_{2n})}	\\
		&\geq\frac{[B^{2n_k}]}{(2[B^{2n_k}]+3)(4[B^{2n_k}]+3)q_{2n}^2}.
	\end{align*}
	
	\subsubsection*{\textbf{\emph{Case (a).}}} If $n=n_k$ for some $k\geq k_0$, then we assume the gap between $\{n_k\}$ is large enough.
	By (\ref{mu1}), we have
	\begin{align*}
		&\ \mu(J(1,\sigma_2,\ldots,1,\sigma_{2n_k}))	\\
		=&\ \prod_{j=1}^{k}\frac{1}{[B^{2n_j}]}\left(\frac{1}{B^{2m_j}q_{2m_j}^2}\right)^{s_{m_j,B}(\alpha)}	\\
		\leq&\ \prod_{j=1}^{k_0}\frac{1}{B^{2n_j}}\left(\frac{1}{B^{2m_j}q_{2m_j}^2}\right)^{s_{m_j,B}(\alpha)}\cdot\prod_{j=k_0+1}^{k}\frac{1}{B^{2n_j}}\left(\frac{1}{B^{2m_j}q_{2m_j}^2}\right)^{s_{m_j,B}(\alpha)}	\\
		\leq&\ c_I\prod_{j=1}^{k_0}\frac{1}{B^{2n_j}}\left(\frac{1}{B^{2n_j}q_{2n_j}^2}\right)^{t}\cdot\prod_{j=k_0+1}^{k}2\cdot\left(\frac{1}{B^{2m_j+2n_j}q_{2m_j}^2}\right)^{t+\varepsilon}	\\
		\leq&\ c_I\prod_{j=1}^{k_0}\left(\frac{1}{B^{4n_j}q_{2n_j}^2}\right)^{t}\cdot\prod_{j=k_0+1}^{k}\left(\frac{1}{B^{4n_j}q_{2m_j}^2}\right)^{t}	\\
		\leq&\ c_I\prod_{j=1}^{k}\left(\frac{1}{B^{4n_j}q_{2m_j}^2}\right)^{t}	\\
		\leq&\ c_I\cdot2^{6kt}\cdot\frac{1}{q_{2n_k}^{2t}}	\\
		\leq&\ c_I\cdot2^{6k}\cdot\left(\frac{3(2\alpha+3)}{\alpha}\right)^t\cdot\abs{J(1,\sigma_2,\ldots,1,\sigma_{2n_k})}^t	\\
		\leq&\ c_I\cdot\abs{J(1,\sigma_2,\ldots,1,\sigma_{2n_k})}^{t-\varepsilon}.
	\end{align*}
	
	\subsubsection*{\textbf{\emph{Case (b).}}} If $n=n_k-1$ for some $k>k_0$. Similarly, for any $\sigma_{2n_k}\in\mathbb{N}$ such that $(\sigma_{2},\ldots,\sigma_{2(n_k-1)},\sigma_{2n_k})\in D_{n_k}$, by (\ref{mu2}) we have
	\begin{align*}
		&\ \mu(J(1,\sigma_{2},\ldots,1,\sigma_{2(n_k-1)}))	\\
		=&\ [B^{2n_k}]\prod_{j=1}^{k}\frac{1}{[B^{2n_j}]}\left(\frac{1}{B^{2m_j}q_{2m_j}^2}\right)^{s_{m_j,B}(\alpha)}	\\
		\leq&\ B^{2n_k}\prod_{j=1}^{k}\frac{1}{B^{2n_j}}\left(\frac{1}{B^{2m_j}q_{2m_j}^2}\right)^{s_{m_j,B}(\alpha)}	\\
		=&\ \prod_{j=1}^{k-1}\frac{1}{B^{2n_j}}\left(\frac{1}{B^{2m_j}q_{2m_j}^2}\right)^{s_{m_j,B}(\alpha)}\cdot\left(\frac{1}{B^{2m_k}q_{2m_k}^2}\right)^{s_{m_k,B}(\alpha)}	\\
		\leq&\ c_I\cdot2^{6(k-1)t}\cdot\frac{1}{q_{2n_{k-1}}^{2t}}\cdot\frac{1}{B^{2m_k(t+\varepsilon)}}\cdot\frac{1}{q_{2m_k}^{2(t+\varepsilon)}}	\\
		\leq&\ c_I\cdot2^{6kt}\cdot\frac{1}{B^{2n_kt}}\cdot\frac{1}{q_{2(n_k-1)}^{2t}}	\\
		\leq&\ c_I\cdot2^{6kt}\cdot\frac{1}{B^{2n_kt}}\cdot\left(\frac{(2[B^{2n_k}]+3)(4[B^{2n_k}]+3)}{[B^{2n_k}]}\right)^t\cdot\abs{J(1,\sigma_2,\ldots,1,\sigma_{2n})}^t	\\
		\leq&\ c_I\cdot\abs{J(1,\sigma_2,\ldots,1,\sigma_{2n})}^{t-\varepsilon}.
	\end{align*}
	
	\subsubsection*{\textbf{\emph{Case (c).}}} If $n_{k-1}<n<n_k-1$ for some $k>k_0$. Then let $l=n-n_{k-1}$ and $l^\prime=n_k-n-1$. By the definition of $\mu$, and by (\ref{mu3}) we have
	\begin{align*}
		&\ \mu(J(1,\sigma_2,\ldots,1,\sigma_{2n}))	\\
		=&\ \sum_{\sigma_{2(n+1)},\ldots,\sigma_{2(n_k-1)}\leq\alpha}[B^{2n_k}]\prod_{j=1}^{k}\frac{1}{[B^{2n_j}]}\left(\frac{1}{B^{2m_j}q_{2m_j}^2}\right)^{s_{m_j,B}(\alpha)}	\\
		\leq&\ \prod_{j=1}^{k-1}\frac{1}{B^{2n_j}}\left(\frac{1}{B^{2m_j}q_{2m_j}^2}\right)^{s_{m_j,B}(\alpha)}\cdot\sum_{\sigma_{2(n+1)},\ldots,\sigma_{2(n_k-1)}\leq\alpha}\left(\frac{1}{B^{2m_k}q_{2m_k}^2}\right)^{s_{m_k,B}(\alpha)}	\\
		\leq&\ c_I\cdot2^{6k}\cdot\frac{1}{q_{2n_{k-1}}^{2t}}\cdot\frac{1}{q_{2l}^{2s_{m_k,B}(\alpha)}}\cdot\sum_{a_2,\ldots,a_{2l^\prime}\leq\alpha}\left(\frac{1}{B^{2m_k}q_{2l^\prime}^2}\right)^{s_{m_k,B}(\alpha)}	\\
		\leq&\ c_I\cdot2^{6k}\cdot\frac{1}{q_{2n}^{2t}}\cdot\sum_{a_2,\ldots,a_{2l^\prime}\leq\alpha}\left(\frac{1}{B^{2m_k}q_{2l^\prime}^2}\right)^{s_{m_k,B}(\alpha)}.
	\end{align*}
	By the definition of $s_{m_k,B}(\alpha)$, we have 
	\begin{align*}
		1&=\sum_{a_2,\ldots,a_{2m_k}\leq\alpha}\left(\frac{1}{B^{2m_k}q_{2m_k}^2}\right)^{s_{m_k,B}(\alpha)}	\\
		&\geq\frac{1}{2^{2s_{m_k,B}(\alpha)}}\cdot\sum_{a_2,\ldots,a_{2l}\leq\alpha}\left(\frac{1}{B^{2l}q_{2l}^2}\right)^{s_{m_k,B}(\alpha)}\cdot\sum_{a_{2},\ldots,a_{2l^\prime}\leq\alpha}\left(\frac{1}{B^{2l^\prime}q_{2l^\prime}^2}\right)^{s_{m_k,B}(\alpha)}.
	\end{align*}
	Choose $N_0\in\mathbb{N}$ such that	for any $m_k,l> N_0$, we have
	\begin{equation*}
		\abs{s_{m_k,B}(\alpha)-s_{l,B}(\alpha)}<\delta,\ \text{where}\ \delta=\frac{\log2\cdot\varepsilon}{2\log B+\log2+\log(\alpha+1)}.
	\end{equation*}
	If $l\leq N_0$, then we have 
	\begin{align*}
		&\ \mu(J(1,\sigma_2,\ldots,1,\sigma_{2n}))	\\
		\leq&\ c_I\cdot2^{6k}\cdot\frac{1}{q_{2n}^{2t}}\cdot2^{s_{m_k,B}(\alpha)}\cdot\left(\sum_{a_2,\ldots,a_{2l}\leq\alpha}\left(\frac{1}{B^{2l}q_{2l}^2}\right)^{s_{m_k,B}(\alpha)}\right)^{-1}	\\
		\leq&\ c_I\cdot2^{6k}\cdot\frac{1}{q_{2n}^{2t}}\cdot\frac{B^{2l}2^{2l}(\alpha+1)^{2l}}{\alpha^l}	\\
		\leq&\ c_I\cdot2^{6k}\cdot\frac{B^{2l}2^{2l}(\alpha+1)^{2l}}{\alpha^l}\cdot\left(\frac{3(2\alpha+3)}{\alpha}\right)^t\cdot\abs{J(1,\sigma_2,\ldots,1,\sigma_{2n})}^t	\\
		\leq&\ c_I\cdot\abs{J(1,\sigma_2,\ldots,1,\sigma_{2n})}^{t-\varepsilon}.
	\end{align*}
	If $l>N_0$, then we have 
	\begin{align*}
		&\ \mu(J(1,\sigma_2,\ldots,1,\sigma_{2n}))	\\
		\leq&\ c_I\cdot2^{6k}\cdot\frac{1}{q_{2n}^{2t}}\cdot2^{s_{m_k,B}(\alpha)}\cdot\left(\sum_{a_2,\ldots,a_{2l}\leq\alpha}\left(\frac{1}{B^{2l}q_{2l}^2}\right)^{s_{m_k,B}(\alpha)}\right)^{-1}	\\
		\leq&\ c_I\cdot2^{6k}\cdot\frac{1}{q_{2n}^{2t}}\cdot2^{s_{m_k,B}(\alpha)}\cdot\left(\sum_{a_2,\ldots,a_{2l}\leq\alpha}\left(\frac{1}{B^{2l}q_{2l}^2}\right)^{s_{m_l,B}(\alpha)+\delta}\right)^{-1}	\\
		\leq&\ c_I\cdot2^{6k}\cdot\frac{4}{q_{2n}^{2t}}\cdot(B^{2l}2^l(\alpha+1)^l)^\delta	\\
		\leq&\ c_I\cdot2^{6k}\cdot4\cdot(B^{2l}2^l(\alpha+1)^l)^\delta\cdot\left(\frac{3(2\alpha+3)}{\alpha}\right)^t\cdot\abs{J(1,\sigma_2,\ldots,1,\sigma_{2n})}^t	\\
		\leq&\ c_I\cdot\abs{J(1,\sigma_2,\ldots,1,\sigma_{2n})}^{t-\varepsilon}.
	\end{align*}
	To sum up, our desired result is proved.
\end{proof}

\subsection{Lower bound of Hausdorff dimension of $F(B)$}

In this subsection, we use the mass distribution principle to estimate the lower bound of $\dim_H F_\alpha(B)$, and then get a lower bound of $\dim_H F(B)$. We will proceed in two steps.

\subsubsection{Gaps between the adjoint fundamental intervals}

Given a fundamental interval $J(1,\sigma_2,\ldots,1,\sigma_{2n})$ where $(\sigma_{2},\ldots,\sigma_{2n})\in D_n$, denote $g^r(1,\sigma_{2},\ldots,1,\sigma_{2n})$ the distance between $J(1,\sigma_2,\ldots,1,\sigma_{2n})$ and the fundamental interval of order $2n$ which lies on the right of $J(1,\sigma_2,\ldots,1,\sigma_{2n})$ and is closest to it, $g^l(1,\sigma_{2},\ldots,1,\sigma_{2n})$ the distance between $J(1,\sigma_2,\ldots,1,\sigma_{2n})$ and the fundamental interval of order $2n$ which lies on the left of $J(1,\sigma_2,\ldots,1,\sigma_{2n})$ and is closest to it.

Let $\widetilde{g}(1,\sigma_{2},\ldots,1,\sigma_{2n})=\min\{g^r(1,\sigma_{2},\ldots,1,\sigma_{2n}),g^l(1,\sigma_{2},\ldots,1,\sigma_{2n})\}$.

\begin{prop}[Gaps in $F_\alpha(B)$]\label{gap}
	For any $B>1$ and $\alpha\geq2$, if $n\ne n_k-1$ for any $k>1$, we have
	\begin{equation}\label{gap1}
		\widetilde{g}(1,\sigma_{2},\ldots,1,\sigma_{2n})\geq\frac{2}{\alpha}\cdot\abs{J(1,\sigma_2,\ldots,1,\sigma_{2n})},
	\end{equation}
	and if $n=n_k-1$ for some $k>1$, we have 
	\begin{equation}\label{gap2}
		\widetilde{g}(1,\sigma_{2},\ldots,1,\sigma_{2n})\geq2\cdot\abs{J(1,\sigma_2,\ldots,1,\sigma_{2n})}.
	\end{equation}
\end{prop}

\begin{proof}
	We fix a sequence of integers $\{n_k\}_{k\geq1}$ with gaps large enough. For $(\sigma_{2},\ldots,\sigma_{2n})\in D_n$, we estimate the gaps under several different conditions as follows.
	\subsubsection*{\textbf{\emph{Case (a).}}} If $n\ne n_k-1$ for any $k>1$, by the definition of fundamental intervals (\ref{c1}), we have \begin{equation*}
		J(1,\sigma_{2},\ldots,1,\sigma_{2n})=\bigcup_{1\leq\sigma_{2(n+1)}\leq\alpha}\textnormal{cl}I(1,\sigma_{2},\ldots,1,\sigma_{2n},1,\sigma_{2(n+1)}),
	\end{equation*}
	and by (\ref{c2}) we have 
	\begin{equation*}
		\abs{J(1,\sigma_{2},\ldots,1,\sigma_{2n})}=\frac{\alpha}{(q_{2n+1}+q_{2n})((\alpha+1)q_{2n+1}+q_{2n})}.
	\end{equation*}
	Note that $g^l(1,\sigma_{2},\ldots,1,\sigma_{2n})$ is the distance between the left endpoint of $J(1,\sigma_{2},\ldots,1,\sigma_{2n})$ and the right endpoint of $J(1,\sigma_{2},\ldots,1,\sigma_{2n}-1)$ (if it exists). Hence, 
	\begin{align*}
		g^l(1,\sigma_{2},\ldots,1,\sigma_{2n})&=[0;1,\sigma_{2},\ldots,1,\sigma_{2n},1,1]-[0;1,\sigma_{2},\ldots,1,\sigma_{2n}-1,1,\alpha+1]	\\
		&=\frac{2p_{2n}+p_{2n-1}}{2q_{2n}+q_{2n-1}}-\frac{(\alpha+2)p_{2n}-p_{2n-1}}{(\alpha+2)q_{2n}-q_{2n-1}}	\\
		&=\frac{\alpha+4}{(2q_{2n}+q_{2n-1})((\alpha+2)q_{2n}-q_{2n-1})}	\\
		&=\frac{\alpha+4}{(q_{2n+1}+q_{2n})((\alpha+2)q_{2n}-q_{2n-1})}.
	\end{align*}
	The value of $g^r(1,\sigma_{2},\ldots,1,\sigma_{2n})$ is the distance between the left endpoint of $J(1,\sigma_{2},\ldots,1,\sigma_{2n}+1)$ (if it exists) and the right endpoint of $J(1,\sigma_{2},\ldots,1,\sigma_{2n})$. Hence, 
	\begin{align*}
		g^r(1,\sigma_{2},\ldots,1,\sigma_{2n})&=[0;;1,\sigma_{2},\ldots,1,\sigma_{2n}+1,1,1]-[0;1,\sigma_{2},\ldots,1,\sigma_{2n},1,\alpha+1]	\\
		&=\frac{2p_{2n}+3p_{2n-1}}{2q_{2n}+3q_{2n-1}}-\frac{(\alpha+2)p_{2n}+(\alpha+1)p_{2n-1}}{(\alpha+2)q_{2n}+(\alpha+1)q_{2n-1}}	\\
		&=\frac{\alpha+4}{(2q_{2n}+3q_{2n-1})((\alpha+2)q_{2n}+(\alpha+1)q_{2n-1})}	\\
		&=\frac{\alpha+4}{(2q_{2n}+3q_{2n-1})((\alpha+1)q_{2n+1}+q_{2n})}.
	\end{align*}
	Thus, 
	\begin{align}
		\widetilde{g}(1,\sigma_{2},\ldots,1,\sigma_{2n})&=g^r(1,\sigma_{2},\ldots,1,\sigma_{2n})	\notag	\\
		&=\frac{\alpha+4}{(2q_{2n}+3q_{2n-1})((\alpha+1)q_{2n+1}+q_{2n})}	\notag	\\
		&=\frac{\alpha+4}{\alpha}\cdot\frac{2q_{2n}+q_{2n-1}}{2q_{2n}+3q_{2n-1}}\cdot\abs{J(1,\sigma_2,\ldots,1,\sigma_{2n})}	\notag	\\
		&\geq\frac{1}{3}\cdot\frac{\alpha+4}{\alpha}\abs{J(1,\sigma_2,\ldots,1,\sigma_{2n})}	\notag	\\
		&\geq\frac{2}{\alpha}\cdot\abs{J(1,\sigma_2,\ldots,1,\sigma_{2n})}.	\label{gapa1}
	\end{align}
	\subsubsection*{\textbf{\emph{Case (b).}}} If $n=n_k-1$ for some $k>1$, by the definition of fundamental intervals (\ref{c1}), we have
	\begin{equation*}
		J(1,\sigma_{2},\ldots,1,\sigma_{2n})=\bigcup_{[B^{2n_k}]+1\leq\sigma_{2(n+1)}\leq2[B^{2n_k}]}\textnormal{cl}I(1,\sigma_{2},\ldots,1,\sigma_{2n},1,\sigma_{2(n+1)}),
	\end{equation*}
	and by (\ref{c3}) we have 
	\begin{equation*}
		\abs{J(1,\sigma_{2},\ldots,1,\sigma_{2n})}=\frac{[B^{2n_k}]}{(([B^{2n_k}]+1)q_{2n+1}+1_{2n})((2[B^{2n_k}]+1)q_{2n+1}+q_{2n})}.
	\end{equation*}
	Note that $g^l(1,\sigma_{2},\ldots,1,\sigma_{2n})$ is the distance between the left endpoint of $J(1,\sigma_{2},\ldots,1,\sigma_{2n})$ and the right endpoint of $J(1,\sigma_{2},\ldots,1,\sigma_{2n}-1)$ (if it exists). Hence,
	\begin{align*}
		&\ g^l(1,\sigma_{2},\ldots,1,\sigma_{2n})	\\
		=&\ [0;1,\sigma_{2},\ldots,1,\sigma_{2n},1,[B^{2n_k}]+1]-[0;1,\sigma_{2},\ldots,1,\sigma_{2n}-1,1,2[B^{2n_k}]+1]	\\
		=&\ \frac{([B^{2n_k}]+2)p_{2n}+([B^{2n_k}]+1)p_{2n-1}}{([B^{2n_k}]+2)q_{2n}+([B^{2n_k}]+1)q_{2n-1}}-\frac{(2[B^{2n_k}]+2)p_{2n}-p_{2n-1}}{(2[B^{2n_k}]+2)q_{2n}-q_{2n-1}}	\\
		=&\ \frac{2[B^{2n_k}]^2+5[B^{2n_k}]+4}{(([B^{2n_k}]+2)q_{2n}+([B^{2n_k}]+1)q_{2n-1})((2[B^{2n_k}]+2)q_{2n}-q_{2n-1})}.
	\end{align*}
	Observe that the value of $g^r(1,\sigma_{2},\ldots,1,\sigma_{2n})$ is the distance between the left endpoint of $J(1,\sigma_{2},\ldots,1,\sigma_{2n}+1)$ (if it exists) and the right endpoint of $J(1,\sigma_{2},\ldots,1,\sigma_{2n})$. Hence, 
	\begin{align*}
		&g^r(1,\sigma_{2},\ldots,1,\sigma_{2n})	\\
		=&\ [0;1,\sigma_{2},\ldots,1,\sigma_{2n}+1,1,[B^{2n_k}]+1]-[0;1,\sigma_{2},\ldots,1,\sigma_{2n},1,2[B^{2n_k}]+1]	\\
		=&\ \frac{([B^{2n_k}]+2)p_{2n}+(2[B^{2n_k}]+3)p_{2n-1}}{([B^{2n_k}]+2)q_{2n}+(2[B^{2n_k}]+3)q_{2n-1}}-\frac{((2[B^{2n_k}]+2)p_{2n}+(2[B^{2n_k}]+1)p_{2n-1})}{((2[B^{2n_k}]+2)q_{2n}+(2[B^{2n_k}]+1)q_{2n-1})}	\\
		=&\ \frac{2[B^{2n_k}]^2+5[B^{2n_k}]+4}{(([B^{2n_k}]+2)q_{2n}+(2[B^{2n_k}]+3)q_{2n-1})((2[B^{2n_k}]+2)q_{2n}+(2[B^{2n_k}]+1)q_{2n-1})}	\\
		=&\ \frac{2[B^{2n_k}]^2+5[B^{2n_k}]+4}{(([B^{2n_k}]+2)q_{2n}+(2[B^{2n_k}]+3)q_{2n-1})((2[B^{2n_k}]+1)q_{2n+1}+q_{2n})}.
	\end{align*}
	Thus,
	\begin{align}
		&\ \widetilde{g}(1,\sigma_{2},\ldots,1,\sigma_{2n})=g^r(1,\sigma_{2},\ldots,1,\sigma_{2n})	\notag	\\
		=&\ \frac{2[B^{2n_k}]^2+5[B^{2n_k}]+4}{(([B^{2n_k}]+2)q_{2n}+(2[B^{2n_k}]+3)q_{2n-1})((2[B^{2n_k}]+1)q_{2n+1}+q_{2n})}	\notag	\\
		\geq&\ \frac{1}{3}\cdot\frac{2[B^{2n_k}]^2+5[B^{2n_k}]+4}{[B^{2n_k}]}\cdot\abs{J(1,\sigma_2,\ldots,1,\sigma_{2n})}	\notag	\\
		\geq&\ 2\cdot\abs{J(1,\sigma_{2},\ldots,1,\sigma_{2n})}.	\label{gapb1}
	\end{align}
	\subsubsection*{\textbf{\emph{Exceptional cases.}}} If $n\ne n_k-1$ and $\sigma_{2n}=[B^{2n}]+1$, we can extend our definition of fundamental intervals to $J(1,\sigma_{2},\ldots,1,\sigma_{2n},1,\sigma_{2n}-1)$. Following our calculation in previous cases, we immediately obtain 
	\begin{equation}\label{gape1}
		\widetilde{g}(1,\sigma_2,\ldots,1,\sigma_{2n})=g^r(1,\sigma_{2},\ldots,1,\sigma_{2n})=\frac{\alpha+4}{(2q_{2n}+3q_{2n-1})((\alpha+1)q_{2n+1}+q_{2n})}.
	\end{equation}
	
	If $n\ne n_k-1$ and $\sigma_{2n}=1$, $g^l(1,\sigma_2,\ldots,1,\sigma_{2n})$ is larger than the distance between the left endpoint of $I(1,\sigma_2,\ldots,1,\sigma_{2(n-1)},1)$ and the left endpoint of $I(1,\sigma_{2},\ldots,1,\sigma_{2(n-1)})$. Hence,
	\begin{align*}
		g^l(1,\sigma_2,\ldots,1,\sigma_{2n})&\geq[0;1,\sigma_{2},\ldots,1,\sigma_{2(n-1)},1,1]-[0;1,\sigma_{2},\ldots,1,\sigma_{2(n-1)}]	\\
		&=\frac{1}{(q_{2n-1}+q_{2n-2})q_{2n-2}},
	\end{align*}
	while 
	\begin{equation*}
		g^r(1,\sigma_2,\ldots,1,\sigma_{2n})=\frac{\alpha+4}{(2q_{2n}+3q_{2n-1})((\alpha+1)q_{2n+1}+q_{2n})}<g^l(1,\sigma_2,\ldots,1,\sigma_{2n})
	\end{equation*}
	Thus, we have 
	\begin{equation}\label{gape2}
		\widetilde{g}(1,\sigma_2,\ldots,1,\sigma_{2n})=\frac{\alpha+4}{(2q_{2n}+3q_{2n-1})((\alpha+1)q_{2n+1}+q_{2n})},
	\end{equation}
	which is exactly equal to $g^l(1,\sigma_2,\ldots,1,\sigma_{2n}+1)$ and strictly larger than $g^r(1,\sigma_2,\ldots,1,\sigma_{2n}+1)=\widetilde{g}(1,\sigma_2,\ldots,1,\sigma_{2n}+1)$.
	
	If $n\ne n_k-1$ and $\sigma_{2n}=2[B^{2n}]\ \text{or}\ \alpha$, $g^r(1,\sigma_2,\ldots,1,\sigma_{2n})$ is larger than the distance between left endpoint of $I(1,\sigma_2,\ldots,1,\sigma_{2(n-1)}+1,1)$ and the right endpoint of $I(1,\sigma_{2},\ldots,1,\sigma_{2(n-1)},1)$. Hence, 
	\begin{align*}
		g^r(1,\sigma_2,\ldots,1,\sigma_{2n})&\geq[0;1,\sigma_{2},\ldots,\sigma_{2(n-1)}+1,1,1]-[0;1,\sigma_{2},\ldots,\sigma_{2(n-1)},1]	\\
		&=\frac{1}{(2q_{2n}+3q_{2n-3})q_{2n-1}},
	\end{align*}
	while 
	\begin{equation*}
		g^l(1,\sigma_{2},\ldots,1,\sigma_{2n})=\frac{\alpha+4}{(2q_{2n}+q_{2n-1})((\alpha+2)q_{2n}-q_{2n-1})}<g^r(1,\sigma_{2},\ldots,1,\sigma_{2n}).
	\end{equation*}
	Thus, we have
	\begin{equation}\label{gape3}
		\widetilde{g}(1,\sigma_{2},\ldots,1,\sigma_{2n})=\frac{\alpha+4}{(2q_{2n}+q_{2n-1})((\alpha+2)q_{2n}-q_{2n-1})},
	\end{equation}
	which is exactly equal to $g^r(1,\sigma_{2},\ldots,1,\sigma_{2n}-1)$ and $\widetilde{g}(1,\sigma_{2},\ldots,1,\sigma_{2n}-1)$.
	
	If $n=n_k-1$ and $\sigma_{2n}=1$, $g^l(1,\sigma_{2},\ldots,1,\sigma_{2n})$ is larger than the distance between left endpoint of $I(1,\sigma_{2},\ldots,1,\sigma_{2(n-1)},1)$ and left endpoint of $I(1,\sigma_{2},\ldots,1,\sigma_{2(n-1)})$. Hence,
	\begin{align*}
		g^l(1,\sigma_{2},\ldots,1,\sigma_{2n})&\geq[0;1,\sigma_{2},\ldots,1,\sigma_{2(n-1)}+1,1,1]-[0;1,\sigma_{2},\ldots,1,\sigma_{2(n-1)}]	\\
		&=\frac{1}{(q_{2n-1}+q_{2n-2})q_{2n-2}},
	\end{align*}
	while 
	\begin{align*}
		g^r(1,\sigma_{2},\ldots,1,\sigma_{2n})&=\frac{2[B^{2n_k}]^2+5[B^{2n_k}]+4}{(([B^{2n_k}]+2)q_{2n}+(2[B^{2n_k}]+3)q_{2n-1})((2[B^{2n_k}]+1)q_{2n+1}+q_{2n})}	\\
		&<g^l(1,\sigma_{2},\ldots,1,\sigma_{2n}).
	\end{align*}
	Thus, we have 
	\begin{align}
		&\ \widetilde{g}(1,\sigma_{2},\ldots,1,\sigma_{2n})=g^r(1,\sigma_{2},\ldots,1,\sigma_{2n})	\notag	\\
		=&\ \frac{2[B^{2n_k}]^2+5[B^{2n_k}]+4}{(([B^{2n_k}]+2)q_{2n}+(2[B^{2n_k}]+3)q_{2n-1})((2[B^{2n_k}]+1)q_{2n+1}+q_{2n})},	\label{gape4}
	\end{align}
	which is exactly equal to $g^l(1,\sigma_2,\ldots,1,\sigma_{2n}+1)$ and strictly larger than $g^r(1,\sigma_2,\ldots,1,\sigma_{2n}+1)=\widetilde{g}(1,\sigma_2,\ldots,1,\sigma_{2n}+1)$.
	
	If $n=n_k-1$ and $\sigma_{2n}=\alpha$, $g^r(1,\sigma_2,\ldots,1,\sigma_{2n})$ is larger than the distance between left endpoint of $I(1,\sigma_2,\ldots,1,\sigma_{2(n-1)}+1,1)$ and the right endpoint of $I(1,\sigma_{2},\ldots,1,\sigma_{2(n-1)},1)$. Hence, 
	\begin{align*}
		g^r(1,\sigma_2,\ldots,1,\sigma_{2n})&\geq[0;1,\sigma_{2},\ldots,\sigma_{2(n-1)}+1,1,1]-[0;1,\sigma_{2},\ldots,\sigma_{2(n-1)},1]	\\
		&=\frac{1}{(2q_{2n}+3q_{2n-3})q_{2n-1}},
	\end{align*}
	while 
	\begin{align*}
		g^l(1,\sigma_2,\ldots,1,\sigma_{2n})&=\frac{2[B^{2n_k}]^2+5[B^{2n_k}]+4}{(([B^{2n_k}]+2)q_{2n}+([B^{2n_k}]+1)q_{2n-1})((2[B^{2n_k}]+2)q_{2n}-q_{2n-1})}	\\
		&\leq g^r(1,\sigma_2,\ldots,1,\sigma_{2n}).
	\end{align*}
	Thus, we have 
	\begin{align}
		&\ \widetilde{g}(1,\sigma_{2},\ldots,1,\sigma_{2n})=g^l(1,\sigma_{2},\ldots,1,\sigma_{2n})	\notag	\\
		=&\ \frac{2[B^{2n_k}]^2+5[B^{2n_k}]+4}{(([B^{2n_k}]+2)q_{2n}+([B^{2n_k}]+1)q_{2n-1})((2[B^{2n_k}]+2)q_{2n}-q_{2n-1})}	\label{gape5}
	\end{align}
	which is exactly equal to $g^r(1,\sigma_{2},\ldots,1,\sigma_{2n}-1)$ and $\widetilde{g}(1,\sigma_{2},\ldots,1,\sigma_{2n}-1)$.
	
	By (\ref{gapa1})-(\ref{gape5}), if $n\ne n_k-1$ for any $k>1$, we have
	\begin{equation}
		\widetilde{g}(1,\sigma_{2},\ldots,1,\sigma_{2n})\geq\frac{2}{\alpha}\cdot\abs{J(1,\sigma_2,\ldots,1,\sigma_{2n})},
	\end{equation}
	if $n=n_k-1$ for some $k>1$, we have 
	\begin{equation}
		\widetilde{g}(1,\sigma_{2},\ldots,1,\sigma_{2n})\geq2\cdot\abs{J(1,\sigma_2,\ldots,1,\sigma_{2n})},
	\end{equation}
	and therefore we finish our proof.
\end{proof}

\subsubsection{Calculation of the lower bound}

For any $n\geq1$ and $(\sigma_{2},\ldots,\sigma_{2n})\in D_n$, if $n\ne n_k-1$ for any $k>1$, we define 
\begin{align}
	g(1,\sigma_{2},\ldots,1,\sigma_{2n})&\coloneqq\frac{2}{\alpha}\cdot\abs{J(1,\sigma_{2},\ldots,1,\sigma_{2n})}	\notag	\\
	&=\frac{2}{(q_{2n+1}+q_{2n})((\alpha+1)q_{2n+1}+q_{2n})},	\label{G1}
\end{align}
and if $n=n_k-1$ for some $k>1$, we define
\begin{align}
	g(1,\sigma_{2},\ldots,1,\sigma_{2n})&\coloneqq2\cdot\abs{J(1,\sigma_{2},\ldots,1,\sigma_{2n})}	\notag	\\
	&=\frac{2[B^{2n_k}]}{(([B^{2n_k}]+1)q_{2n+1}+q_{2n})((2[B^{2n_k}]+1)q_{2n+1}+q_{2n})}.	\label{G2}
\end{align}
The notation $g$ gives a lower bound of the distance between fundamental intervals.

\begin{prop}[Lower bound of $\dim_H F_\alpha(B)$]\label{lower}
	Under the same assumption as Proposition \ref{estmu}, we have 
	\begin{equation*}
		\dim_H F_\alpha(B)\geq s_B(\alpha)-4\varepsilon.
	\end{equation*}
\end{prop}

\begin{proof}
	Take $r_0=\min_{1\leq j\leq n_{k_0}}\min_{(\sigma_{2},\ldots,\sigma_{2j})\in D_j}g(1,\sigma_{2},\ldots,1,\sigma_{2j})$. For any $x\in F_\alpha(B)$ and $0<r<r_0$, there exists a unique sequence $\sigma_{2},\sigma_{4},\ldots$ such that $x\in J(1,\sigma_{2},\ldots,1,\sigma_{2k})$ for all $k\geq1$ and for $n\geq n_{k_0}$,
	\begin{equation*}
		g(1,\sigma_{2},\ldots,1,\sigma_{2n},1,\sigma_{2(n+1)})\leq r<g(1,\sigma_{2},\ldots,1,\sigma_{2n}).
	\end{equation*}
	From the definition of $g(1,\sigma_{2},\ldots,1,\sigma_{2n})$, we immediately find that the ball $B(x,r)$ can intersect only one fundamental interval of order $n$, which is $J(1,\sigma_{2},\ldots,1,\sigma_{2n})$.
	
	\subsubsection*{\textbf{\emph{Case (a).}}} If $n=n_k-1$, for some $k \geq k_0$, we start with counting the number of the basic intervals that intersect with the ball $B(x,r)$. Since
	\begin{equation*}
		\abs{I(1,\sigma_2,\ldots,1,\sigma_{2(n+1)})}=\frac{1}{q_{2n+2}(q_{2n+2}+q_{2n+1})}\geq\frac{1}{12[B^{2n_k}]^2q_{2n+1}^2},
	\end{equation*}
	then the ball $B(x,r)$ can intersects at most $\min\{[B^{2n_k}],24r[B^{2n_k}]^2q_{2n+1}^2+2\}$ basic intervals of order $n+1$. We denote $N$ as the number of fundamental intervals of order $n+1$ contained in $J(1,\sigma_{2},\ldots,1,\sigma_{2n})$ that intersects the ball $B(x,r)$. Therefore, 
	\begin{align*}
		\mu(B(x,r))&\leq N\cdot\mu(I(1,\sigma_{2},\ldots,1,\sigma_{2(n+1)}))	\\
		&=N\cdot\mu(J(1,\sigma_{2},\ldots,1,\sigma_{2(n+1)}))	\\
		&=\frac{N}{[B^{2n_k}]}\cdot\mu(J(1,\sigma_{2},\ldots,1,\sigma_{2n}))	\\
		&\leq\frac{N}{[B^{2n_k}]}\cdot c_I\cdot\abs{J(1,\sigma_{2},\ldots,1,\sigma_{2n})}^{t-\varepsilon}	\\
		&=\frac{N}{[B^{2n_k}]}\cdot c_I\cdot\left(\frac{[B^{2n_k}]}{(([B^{2n_k}]+1)q_{2n+1}+q_{2n})((2[B^{2n_k}]+1)q_{2n+1}+q_{2n})}\right)^{t-\varepsilon}	\\
		&\leq\frac{N}{[B^{2n_k}]}\cdot c_I\cdot\left(\frac{1}{2[B^{2n_k}]q_{2n+1}^2}\right)^{t-\varepsilon}.
	\end{align*}
	Noticing that $r\geq g(1,\sigma_{2},\ldots,1,\sigma_{2(n+1)})$, we have $24r[B^{2n_k}]^2q_{2n+1}^2+2\leq (49\alpha+94)r[B^{2n_k}]^2q_{2n+1}^2$. Thus,
	\begin{align}
		\mu(B(x,r))&\leq\frac{N}{[B^{2n_k}]}\cdot c_I\cdot\left(\frac{1}{2[B^{2n_k}]q_{2n+1}^2}\right)^{t-\varepsilon}	\notag	\\
		&\leq\frac{((49\alpha+94)r[B^{2n_k}]^2q_{2n+1}^2)^{t-\varepsilon}[B^{2n_k}]^{1+\varepsilon-t}}{[B^{2n_k}]}\cdot c_I\cdot\left(\frac{1}{2[B^{2n_k}]q_{2n+1}^2}\right)^{t-\varepsilon}	\notag	\\
		&=2^{\varepsilon-t} \cdot c_I\cdot(49\alpha+94)^{t-\varepsilon}\cdot r^{t-\varepsilon}.	\label{low1}
	\end{align}
	
	\subsubsection*{\textbf{\emph{Case (b).}}} If $n\ne n_k-1$ for any $k>1$, we also start with counting the number of the basic intervals that intersect with the ball $B(x,r)$. Since 
	\begin{equation*}
		\abs{I(1,\sigma_2,\ldots,1,\sigma_{2(n+1)})}=\frac{1}{q_{2n+2}(q_{2n+2}+q_{2n+1})}\leq\frac{1}{4\alpha^2q_{2n+1}^2},
	\end{equation*}
	then the ball $B(x,r)$ intersects at most $\min\{\alpha,8r\alpha^2q_{2n+1}^2+2\}$ basic intervals of order $n+1$. We denote $N$ as the number of fundamental intervals of order $n+1$ contained in $J(1,\sigma_{2},\ldots,1,\sigma_{2n})$ that intersects the ball $B(x,r)$.
	
	If $n+2\ne n_k$ for any $k>1$, since $r\geq g(1,\sigma_{2},\ldots,1,\sigma_{2(n+1)})$, we have $8r\alpha^2q_{2n+1}^2+2\leq(2+\alpha)4r\alpha^2q_{2n+1}^2$. Therefore, 
	\begin{align}
		\mu(B(x,r))&\leq N\cdot\mu(I(1,\sigma_{2},\ldots,1,\sigma_{2(n+1)}))	\notag	\\
		&\leq N\cdot c_I\cdot\abs{J(1,\sigma_{2},\ldots,1,\sigma_{2(n+1)})}^{t-\varepsilon}	\notag	\\
		&=N\cdot c_I\cdot\left(\frac{\alpha}{(q_{2n+3}+q_{2n+2})((\alpha+1)q_{2n+3}+q_{2n+2})}\right)^{t-\varepsilon}	\notag	\\
		&\leq N\cdot c_I\cdot\left(\frac{\alpha}{9(\alpha+2)q_{2n+1}^2}\right)^{t-\varepsilon}	\notag	\\
		&\leq((2+\alpha)4r\alpha^2q_{2n+1}^2)^{t-\varepsilon}\cdot\alpha^{1+\varepsilon-t}\cdot c_I\cdot\left(\frac{\alpha}{9(\alpha+2)q_{2n+1}^2}\right)^{t-\varepsilon}	\notag	\\
		&=\left(\frac{4}{9}\right)^{t-\varepsilon}\cdot c_I\cdot \alpha^{1+2(t-\varepsilon)}\cdot r^{t-\varepsilon}.	\label{low2}
	\end{align}
	
	If $n+2=n_k$ for some $k\geq k_0$, since $r\geq g(1,\sigma_{2},\ldots,1,\sigma_{2(n+1)})$, we have $8r\alpha^2q_{2n+1}^2+2\leq112r\alpha^2[B^{2n_k}]q_{2n+1}^2$. Therefore,
	\begin{align}
		\mu(B(x,r))&\leq N\cdot\mu(I(1,\sigma_{2},\ldots,1,\sigma_{2(n+1)}))	\notag	\\
		&\leq N\cdot c_I\cdot\abs{J(1,\sigma_{2},\ldots,1,\sigma_{2(n+1)})}^{t-\varepsilon}	\notag	\\
		&=N\cdot c_I\cdot\left(\frac{[B^{2n_k}]}{(([B^{2n_k}]+1)q_{2n+3}+q_{2n+2})((2[B^{2n_k}]+1)q_{2n+3}+q_{2n+2})}\right)	\notag	\\
		&\leq N\cdot c_I\cdot\left(\frac{1}{12[B^{2n_k}]q_{2n+1}^2}\right)^{t-\varepsilon}	\notag	\\
		&\leq(112r\alpha^2[B^{2n_k}]q_{2n+1}^2)^{t-\varepsilon}\cdot\alpha^{1+\varepsilon-t}\cdot c_I\cdot\left(\frac{1}{12[B^{2n_k}]q_{2n+1}^2}\right)^{t-\varepsilon}	\notag	\\
		&=\left(\frac{28}{3}\right)^{t-\varepsilon}\cdot c_I\cdot\alpha^{1+t-\varepsilon}\cdot r^{t-\varepsilon}.	\label{low3}
	\end{align}
	
	By (\ref{low1})-(\ref{low3}) and the mass distribution principle (see Lemma \ref{dp}), we have
	\begin{equation*}
		\liminf_{r\rightarrow 0}\frac{\log\mu(B(x,r))}{\log r}\geq t-2\varepsilon,
	\end{equation*}
	which immediately follows $\dim_H F_\alpha(B)\geq t-2\varepsilon=s_B(\alpha)-4\varepsilon$.
\end{proof}

Now we can finish the proof of Theorem \ref{main3}.

\begin{proof}[Proof of Theorem \ref{main3}]
	By Proposition \ref{upper}, we have $\dim_H F(B)\leq s_B$. Thus we only need to show that $\dim_H F(B)\geq s_B$. Since $\varepsilon>0$ is arbitrary, we have 
	\begin{equation*}
		\dim_H F(B)\geq\dim_H F_\alpha(B)\geq s_B(\alpha).
	\end{equation*}
	By Lemma \ref{editlem}, we have $\dim_H F(B)=s_B$.
\end{proof}

\begin{rem}\label{rem2}
	During the process of proving our theorem, to some degree, the choice of $\{n_k\}_{k\geq1}$ is relatively free. Let $\mathcal{L}\subset\mathbb{N}$ be an infinite set. Then we also have 
	\begin{equation*}
		\dim_H\{x\in I_{even}\colon a_{2n}(x)\geq B^{2n}\ \text{i.m.}\ n\in\mathcal{L}\}=s_B.
	\end{equation*}
	To prove this claim, we only need to restrict our choice of $n_k$ in $\mathcal{L}$.
\end{rem}

\section{Hausdorff dimension of $E(b,c)$ and $\widetilde{E}(b,c)$}\label{6}

In this section, we focus on the sets
\begin{equation*}
	E(b,c)\coloneqq\{x\in I_{even}\colon a_{2n}(x)\geq c^{b^{2n}}\ \text{i.m.}\ n\}\quad (b,c>1)
\end{equation*}
and 
\begin{equation*}
	\widetilde{E}(b,c)\coloneqq\{x\in I_{even}\colon a_{2n}(x)\geq c^{b^{2n}}\ \text{for all}\ n\geq1\}\quad (b,c>1).
\end{equation*}

\subsection{Upper bound of the Hausdorff dimension of $E(b,c)$}

To estimate the upper bound of $\dim_H E(b,c)$, we construct a series of covers for the target set. Fix $1<d<b$ such that $\alpha>1/(1+d^2)>1/(1+b^2)$. We define two kind of sets
\begin{align*}
	I_{n,x}&\coloneqq B\left(\frac{p_{2n+1}(x)}{q_{2n+1}(x)},q_{2n}(x)^{-(1+d^2)}\right),	\\
	J_{n,x}&\coloneqq B\left(\frac{p_{2n+1}(x)}{q_{2n+1}(x)},q_{2n}(x)^{-2(1+d^2)}\right).
\end{align*}
For $q=1,2,\ldots$, we define
\begin{align*}
	\mathscr{I}_q&\coloneqq\{I_{n,x}\colon q_{2n}(x)=q\geq c^{d^{2n}/3}\},	\\
	\mathscr{J}_q&\coloneqq\{J_{n,x}\colon q_{2n}(x)=q\}.
\end{align*}
Before we build up our cover, we give a lemma first.
\begin{lem}\label{qd2}
	If $x\in E(b,c)$, then for any $1<d<b$ we have
	\begin{equation*}
		q_{2n+2}(x)>\max\{q_{2n}(x)^{d^2},c^{d^{2n+2}}\}
	\end{equation*}
	for infinitely many $n$.
\end{lem}

\begin{proof}
	By contrary, assume that there exists an $N\in\mathbb{N}$, such that for all $n\geq N$, $q_{2n+2}(x)\leq\max\{q_{2n}(x)^{d^2},c^{d^{2n+2}}\}$. By Lemma (\ref{lem1}), we have 
	\begin{equation*}
		a_{2n+2}(x)q_{2n}(x)\leq q_{2n+2}(x)\leq q_{2n}(x)^{d^{2s}}\leq \cdots\leq q_{2N}(x)^{d^{2(n+1-N)}}.
	\end{equation*}
	By the definition of set $E(b,c)$, we can always take $n$ large enough that 
	\begin{equation*}
		c^{b^{2n+2}}\leq a_{2n+2}(x)\leq q_{2N}(x)^{d^{2(n+1-N)}}.
	\end{equation*}
	It follows that 
	\begin{equation*}
		2(n+1)\log b+\log\log c\leq 2(n+1-N)\log d+\log\log q_{2N}(x).
	\end{equation*}
	This contradicts to the fact $b>d$. Thus, the proof is finished.
\end{proof}

\begin{prop}\label{cebc}
	For each $m=1,2,\ldots$, the family $\bigcup_{q=m}^\infty\mathscr{I}_q\cup\bigcup_{q=m}^\infty\mathscr{J}_q$ covers $E(b,c)$.
\end{prop}

\begin{proof}
	Notice that \begin{equation*}
		\abs{x-\frac{p_{2n+1}(x)}{q_{2n+1}(x)}}\leq\frac{1}{q_{2n+1}(x)q_{2n+2}(x)}\leq\frac{1}{q_{2n}(x)^{1+d^2}}
	\end{equation*} holds for those $n$ satisfying $q_{2n+2}(x)>\max\{q_{2n}(x)^{d^2},c^{d^{2n+2}}\}$. Thus, we have $x\in I_{n,x}$ and $I_{n,x}\in\bigcup_{q=m}^\infty\mathscr{I}_q$ provided that $q_{2n+1}(x)\geq q_{2n}(x)\geq\max\{m,c^{d^{2n}/3}\}$.
	
	If $q_{2n}(x)<c^{d^{2n}/3}$, then since $q_{2n+2}(x)>c^{d^{2n+2}}>c^{d^{2n}\cdot(1+2d^2)/3}>q_{2n}(x)^{1+2d^2}$, we have 
	\begin{equation*}
		\abs{x-\frac{p_{2n+1}(x)}{q_{2n+1}(x)}}\leq\frac{1}{q_{2n+1}(x)q_{2n+2}(x)}\leq\frac{1}{q_{2n}(x)^{2(1+d^2)}},
	\end{equation*}
	which shows that $x\in J_{n,x}$ and $J_{n,x}\in\bigcup_{q=m}^\infty\mathscr{J}_q$. Therefore, we finish our proof.
\end{proof}

To estimate the number of intervals occurring in our construction, we need a lemma from \L uczak \cite{MR1464375}.

\begin{lem}[\L uczak \cite{MR1464375}]\label{smk}
	Let $k$ and $m$ be natural numbers and let $S(m,k)$ denote the number of sequences $a_1,a_2,\ldots,a_n$ of natural numbers such that $1\leq n\leq k$ and $\prod_{i=1}^{n}a_i\leq m$. Then $S(m,k)\leq m(2+\log m)^{k-1}$.
\end{lem}

\begin{prop}[Upper bound of $\dim_H E(b,c)$]
	For $b,c>1$, we have $\dim_H E(b,c)\leq\dfrac{1}{1+b^2}$.
\end{prop}

\begin{proof}
	We define a set function $\varLambda_\alpha(\mathscr{I})\coloneqq\sum_{i}\abs{I_i}^\alpha$, where $\mathscr{I}=\{I_1,I_2,\ldots\}$ and $\alpha>0$.
	
	We claim that \begin{equation*}
		\lim_{m\rightarrow\infty}\varLambda_\alpha(\bigcup_{q=m}^\infty\mathscr{I}_q\cup\bigcup_{q=m}^\infty\mathscr{J}_q)=0.
	\end{equation*}
	To prove it, we firstly count the number of the intervals that the collection $\bigcup_{q=1}^m\mathscr{I}_q$ contains. Notice that $c^{d^{2n}/3}\leq q_{2n}(x)=q\leq m$. We write $k=[\frac{1}{2}\log_d(3\log_cm)]$, then $n\leq k$. Since $\prod_{j=1}^na_{2j}(x)\leq q_{2n}(x)\leq m$, by lemma (\ref{smk}), we have that the number of such sequence of $\{a_{2j}(x)\}_{1\leq j\leq n}$ ($1\leq n\leq k$) is not more than $S(m,k)$. Note that $p_{2n}(x)$, $q_{2n}(x)$, and thus the interval $I_{n,x}$ are uniquely determined by $\{a_{2j}(x)\}_{1\leq j\leq n}$. Thus
	\begin{equation*}
		\#\bigcup_{q=1}^m\mathscr{I}_q\leq S(m,k)\leq m(2+\log m)^{\frac{1}{2}\log_d(3\log_cm)}.
	\end{equation*}
	
	Now for a sufficiently large $i\in\mathbb{N}$, the collection $\bigcup_{q=2^{i-1}}^{2^i}\mathscr{I}_q$ contains not more than $2^i(2+i\log2)^{\frac{1}{2}\log_d(3i\log_c2)}$ intervals of lengths at most $2^{1-(i-1)(1+d^2)}$. Thus we have
	\begin{align*}
		\sum_{i=1}^{\infty}\varLambda_\alpha(\bigcup_{q=2^{i-1}}^{2^i}\mathscr{I}_q)&\leq\sum_{i=1}^{\infty}2^i(2+i\log2)^{\frac{1}{2}\log_d(3i\log_c2)}2^{(1-(i-1)(1+d^2))\alpha}	\\
		&=2^{(2+d^2)\alpha}\sum_{i=1}^{\infty}2^{i(1-(1+d^2)\alpha)}(2+i\log2)^{\frac{1}{2}\log_d(3i\log_c2)}.
	\end{align*}
	Since $1-(1+d^2)\alpha\leq0$, the series above converges and we conclude the first part of our desired result 
	\begin{equation*}
		\lim_{m\rightarrow\infty}\varLambda_\alpha(\bigcup_{q=m}^\infty\mathscr{I}_q)\leq\lim_{j\rightarrow\infty}\sum_{i=j}^{\infty}\varLambda_\alpha(\bigcup_{q=2^{i-1}}^{2^i}\mathscr{I}_q)=0.
	\end{equation*}
	
	For the rest half of our claim, since $q_{2n+1}(x)\leq 2q_{2n}(x)\leq2q$, then we assert that $\mathscr{J}_q$ contains at most $2q$ intervals of lengths at most $2q^{-2(1+d^2)}$. Hence, 
	\begin{equation*}
		\sum_{q=1}^\infty\varLambda_\alpha(\mathscr{J}_q)\leq\sum_{q=1}^\infty 2q(2q^{-2(1+d^2)})^\alpha=2^{1+\alpha}\sum_{q=1}^\infty q^{1-2(1+d^2)\alpha}.
	\end{equation*}
	Since $1-2(1+d^2)\alpha<-1$, the series above converges which finish the rest half.
	
	Since $\alpha>{1}/({1+d^2})>{1}/({1+b^2})$ is arbitrary, by definition of Hausdorff dimension and proposition \ref{cebc}, $\dim_H E(b,c)\leq{1}/({1+b^2})$.	
\end{proof}

\subsection{Lower bound of the Hausdorff dimension of $E(b,c)$}

In this subsection, we finish our proof of theorem \ref{main4}. In fact, since $\widetilde{E}(b,c)\subset E(b,c)$, we have $\dim_H \widetilde{E}(b,c)\leq\dim_H E(b,c)$, we estimate the lower bound of $\dim_H \widetilde{E}(b,c)$.

By \cite[Example 4.6]{MR1102677}, if $[0,1]=E_0\supset E_1\supset\cdots\supset E_{k-1}\supset E_k\supset\cdots$ be a decreasing sequence of which is a finite union of disjoint closed intervals containing at least $m_k\geq2$ intervals of $E_{k}$ which are separated by gaps of lengths at least $\varepsilon_k>0$, and if $\varepsilon_k\geq\varepsilon_{k+1}>0$, then the Hausdorff dimension of the intersection of $E_k$ is at least
\begin{equation}\label{mkek}
	\liminf_{k\rightarrow \infty}\frac{\log(m_1\cdots m_{k-1})}{-\log(m_k\varepsilon_k)}.
\end{equation}
We directly use this example to finish our proof.

\begin{proof}[Proof of Theorem \ref{main4}]
	We construct a series of interval
	\begin{equation*}
		E_k\coloneqq\{x\in [0,1]\colon a_{2n+1}(x)=1, a_{2n}(x)=a_{2n}, 0\leq n\leq k,  c^{b^{2k+2}}\leq a_{2k+2}(x)\leq3c^{b^{2k+2}}\},
	\end{equation*}
	for $k\geq1$. Denote $m_k$ as the number of integers between $c^{b^{2k}}$ and $3c^{b^{2k}}$, and $\varepsilon_k$ as the infimum of length of gaps between $k$-level sets. And those intervals are separated by gaps of lengths at least $\varepsilon_k$ satisfying
	\begin{align*}
		\varepsilon_k&\geq \frac{2p_{2k}+p_{2k-1}}{2q_{2k}+q_{2k-1}}-\frac{p_{2k}}{q_{2k}}\geq\frac{1}{3\cdot q_{2k}^2}	\\
		&\geq\frac{1}{3\cdot 2^{2k}\cdot\prod_{j=1}^{k}(a_{2j}+1)^2}	\\
		&\geq\frac{1}{3\cdot 2^{2k}\cdot\prod_{j=1}^k (3\cdot c^{b^{2j}}+1)^2}	\\
		&\geq\frac{1}{3\cdot 2^{2k}\cdot 4^{2k}\cdot c^{2b^{2}(b^{2k}-1)/(b^{2}-1)}}.
	\end{align*}
	By (\ref{mkek}), we have \begin{align*}
		\dim_H \widetilde{E}(b,c)&\geq\liminf_{k\rightarrow \infty}\frac{\log(2^k c^{b^2}\cdots c^{b^{2k}})}{-\log(c^{b^{2k}}\cdot\frac{1}{3}\cdot2^{-6k}\cdot c^{-2b^{2}(b^{2k}-1)/(b^{2}-1)})}	\\
		&=\liminf_{k\rightarrow \infty}\frac{k\log2+({b^{2k}-b^2})\log c}{({b^{2k+2}+b^{2k}-2b^2})\log c+(b^2-1)(\log 3+6k\log 2)}	\\
		&=\frac{1}{1+b^2}.
	\end{align*}
	Therefore, we have ${1}/({1+b^2})\leq\dim_H \widetilde{E}(b,c)\leq\dim_H E(b,c)\leq {1}/({1+b^2})$ and finish our proof.
\end{proof}

\section{Generalization on the results of Hausdorff dimensions}\label{7}

In this section, we generalize our results above in order to give a complete determination on Hausdorff dimension of the set $\{x\in I_{even}\colon a_{2n}(x)\geq\phi(n)\ \text{i.m.}\ n\}$, where $\phi$ is an arbitrary positive function defined on natural numbers.

\begin{proof}[Proof of Theorem \ref{main5}]
	If $B=1$, for any $\varepsilon>0$, $\frac{\log\phi(n)}{2n}\leq\log(1+\varepsilon)$, i.e. $\phi(n)\leq(1+\varepsilon)^{2n}$ holds for $n$ infinitely many times. Let $\mathcal{L}=\{n\colon \phi(n)\leq(1+\varepsilon)^{2n}\}$. Then
	\begin{equation*}
		\{x\in I_{even}\colon a_{2n}(x)\geq(1+\varepsilon)^{2n}\ \textnormal{i.m.}\ n\in\mathcal{L}\}\subset F(\phi).
	\end{equation*}
	By Lemma \ref{lem4}, Theorem \ref{main3}, and Remark \ref{rem2}, we have
	\begin{equation*}
		s_1=\sup_{\varepsilon>0}s_{1+\varepsilon}\leq\dim_H F(\phi)\leq\dim_H I_{even}.
	\end{equation*}
	
	If $1\leq B\leq\infty$, for any $\varepsilon>0$, $\frac{\log\phi(n)}{2n}\leq\log(B+\varepsilon)$, i.e. $\phi(n)\leq(B+\varepsilon)^{2n}$ holds for $n$ infinitely many times. There also exists $N\in\mathbb{N}$, such that for all $n>N$, we have $\phi(n)\geq(B-\varepsilon)^{2n}$. Let $\mathcal{L}=\{n\geq N\colon \phi(n)\leq(B+\varepsilon)^{2n}\}$. Then
	\begin{align*}
		\{x\in I_{even}\colon a_{2n}(x)\geq(B+\varepsilon)^{2n}\ \textnormal{i.m.}\ n\in\mathcal{L}\}&\subset F(\phi) \\
		\subset\{x\in I_{even}&\colon a_{2n}(x)\geq(B-\varepsilon)^{2n}\ \textnormal{i.m.}\ n\in\mathcal{L}\}.
	\end{align*}
	By Lemma \ref{lem4}, Theorem \ref{main3}, and Remark \ref{rem2}, we have
	\begin{equation*}
		s_B=\sup_{\varepsilon>0}s_{B+\varepsilon}\leq\dim_H F(\phi)\leq\sup_{\varepsilon>0}s_{B-\varepsilon}=s_B.
	\end{equation*}
	
	If $B=\infty$, $b=1$, for any $\varepsilon>0$, $\frac{\log\log\phi(n)}{2n}\leq\log(1+\varepsilon)$, i.e. $\phi(n)\leq e^{(1+\varepsilon)^{2n}}$ holds for $n$ infinitely many times. Let $\mathcal{L}=\{n\colon \phi(n)\leq\ e^{(1+\varepsilon)^{2n}}\}$. Then
	\begin{align*}
		\{x\in I_{even}\colon a_{2n}(x)\geq e^{(1+\varepsilon)^{2n}}\ \textnormal{i.m.}\ n\in\mathcal{L}\}\subset F(\phi).
	\end{align*}
	By Theorem \ref{main4}, we have
	\begin{equation*}
		\dim_H F(\phi)\geq \lim_{\varepsilon\rightarrow 0}\frac{1}{1+(1+\varepsilon)^2}=\frac{1}{2}.
	\end{equation*}
	On the other hand, since $B=\infty$, we have for any $B_1>0$, $\frac{\log\phi(n)}{2n}\geq\log B_1$, i.e. $\phi(n)\geq B_1^{2n}$ holds for $n$ ultimately. Then 
	\begin{equation*}
		F(\phi)\subset\{x\in I_{even}\colon a_{2n}(x)\geq B_1^{2n}\ \text{i.m.}\ n\}.
	\end{equation*}
	By Lemma \ref{lem4}, Theorem \ref{main3}, and Remark \ref{rem2}, we have
	\begin{equation*}
		\dim_H F(\phi)\leq\lim_{B_1\rightarrow\infty}s_{B_1}=\frac{1}{2}.
	\end{equation*}
	
	If $B=\infty$, $1<b<\infty$, for any $\varepsilon>0$, $\frac{\log\log\phi(n)}{2n}\leq\log(b+\varepsilon)$, i.e. $\phi(n)\leq e^{(b+\varepsilon)^{2n}}$ holds for $n$ infinitely many times. Let $\mathcal{L}=\{n\colon \phi(n)\leq\ e^{(b+\varepsilon)^{2n}}\}$. Similarly, we have $\phi(n)\geq e^{(b-\varepsilon)^{2n}}$ holds for $n$ ultimately. Hence,
	\begin{align*}
		\{x\in I_{even}\colon a_{2n}(x)\geq e^{(b+\varepsilon)^{2n}}\ \text{i.m.}\ n\in\mathcal{L}\}&\subset F(\phi)    \\
		&\subset\{x\in I_{even}\colon a_{2n}(x)\geq e^{(b-\varepsilon)^{2n}}\ \text{i.m.}\ n\}.
	\end{align*}
	By Theorem \ref{main4}, we have
	\begin{equation*}
		\frac{1}{1+b^2}=\lim_{\varepsilon\rightarrow 0}\frac{1}{1+(b+\varepsilon)^2}\leq\dim_H F(\phi)\leq \lim_{\varepsilon\rightarrow 0}\frac{1}{1+(b-\varepsilon)^2}=\frac{1}{1+b^2}.
	\end{equation*}
	
	If $B=\infty$, $b=\infty$, for any $b_1>0$, $\frac{\log\log\phi(n)}{2n}\geq\log b_1$, i.e. $\phi(n)\geq e^{b_1^{2n}}$ holds for $n$ ultimately, we have
	\begin{equation*}
		F(\phi)\subset\{x\in I_{even}\colon a_{2n}(x)\geq e^{b_1^{2n}}\ \text{i.m.}\ n\}.
	\end{equation*}
	By Theorem \ref{main4}, we have
	\begin{equation*}
		\dim_H F(\phi)\leq\lim_{b_1\rightarrow\infty}\frac{1}{1+b_1^2}=0.
	\end{equation*}
	This completes the proof of the theorem.
\end{proof}

\bibliographystyle{plain}
\bibliography{bibtex}

$\newline$\textsc{Yuefeng Tang, School of Mathematics and Statistics, Wuhan University, Wuhan, Hubei, China}
$\newline$\textit{E-mail: }\texttt{tangyuefeng2001@whu.edu.cn}

\end{document}